\documentclass[english,11pt]{article}

\usepackage{mydefinitions}

\usepackage{tikz}
\usetikzlibrary{decorations.pathmorphing} 
\usetikzlibrary{matrix} 
\usetikzlibrary{arrows} 
\usetikzlibrary{calc} 
\usetikzlibrary{patterns}

\usepackage{pgfplots}
\pgfplotsset{compat=1.13}
\usepgfplotslibrary{fillbetween}

\usepackage{environ}

\newif\ifshowfigure
\showfigurefalse
\NewEnviron{Figure}{
    \ifshowfigure
        \begin{figure}
            \BODY
        \end{figure}
    \fi
}

\begin{document}



\title{
The first passage problem for stable linear delay equations perturbed by power law L\'evy noise 
}

\date{\null}

\author{Michael A.\ H\"ogele\footnote{Departamento de Matem\'aticas, Universidad de los Andes,
Bogot\'a, Colombia; ma.hoegele@uniandes.edu.co}  \ 
~and
~Ilya Pavlyukevich\footnote{Institut f\"ur Stochastik, Friedrich--Schiller--Universit\"at Jena, Ernst--Abbe--Platz 2, 
07743 Jena, Germany; ilya.pavlyukevich@uni-jena.de}
}

\maketitle

\begin{abstract}
This article studies a linear scalar delay differential equation subject to small
multiplicative power tail L\'evy noise. 
We solve the first passage (the Kramers) problem with probabilistic methods 
and discover an asymptotic loss of memory in this non-Markovian system.
Furthermore, the mean exit time increases as the power of the small noise amplitude, whereas
the pre-factor accounts for memory effects. 
In particular, we discover a non-linear 
delay-induced exit acceleration due to a non-normal growth phenomenon. 
Our results are illustrated in the well-known linear delay oscillator  
driven by $\alpha$-stable L\'evy flights.
\end{abstract}




\section{Introduction and main results}

The Kramers problem, that is, 
the escape time and location of a randomly excited deterministic dynamical system 
from the proximity of a stable state at small intensity 
was first stated in the context of physical chemistry in the seminal works of 
\cite{Arrhenius-89}, \cite{Eyring-35} and \cite{Kramers-40}. The solution of this 
classical problem is ubiquitous nowadays 
and has given since crucial insight in many diverse areas such as statistical mechanics, 
insurance mathematics, climatic energy balance models and led to the discovery of more complex dynamical effects 
such as for instance stochastic resonance in complex systems \cite{BenziPSV-81, BenziPSV-82, BenziPSV-83}. 

In the mathematics literature, 
for Markovian systems such as of ordinary and partial differential equations with small Gaussian noise 
this problem was studied extensively with the help of large deviations theory 
which goes back to the seminal work by \cite{Cramer-38} and later by \cite{FW70, Wentzell-76,FreidlinW-12}, and in the 
recent years 
by 
\cite{DeuStr89,DemboZ-98,BarBovMel10,BovEckGayKle02,BovierEGK-04,BerglundG-04,BerGen-10,BerGen-12,CerRoeck-04,Freidlin-00}. 
It is well-known 
that for $\e$-small Brownian diffusion in the potential well, 
the expected exit (excitation) time grows \textit{exponentially} as
\begin{equation}\label{e:GaussExit}
\E \tau^\e \sim \ex^\frac{\bar V}{\e^2}, 
\end{equation}
where $\bar V$ is proportional to the height of the potential barrier  
which has to be overcome by the particle. The exit location is determined by 
the deterministic energy minimizing path. 

\cite{PenlandE-08}, however, discusses ``physical origins of stochastic forcing'' and the 
trade off between Gaussian and non-Gaussian white and colored noises. 
Following their lines of reasoning, modelling with \emph{L\'evy noises} 
is the second best choice in complex systems, since they allow for richer effects
such as 
local asymmetry and the presence of large bursts and jumps which are difficult to realize only with Gaussian influences. 
Furthermore \cite{bodai2017predictability} give evidence for physicality 
of L\'evy noises in particular due to the predictability of fat-tail extremes 
while in the limit of small noise intensity, 
the (re-normalized) exit times are exponentially distributed and hence memoryless or unpredictable. 
Heavy-tailed noise has also been found present in many physical systems, for instance in works by 
\cite{Ditlevsen-99a, Ditlevsen-99b}, \cite{ChenWuDuanKurts-19} and \cite{GaiHoeKosMon17}.

Due to a large variety of L\'evy processes, i.e.\ stochastically continuous processes with independent 
stationary increments, there is no general Kramers' theory for L\'evy driven systems. 
Besides the well studied Gaussian case, 
a large deviations result for heavy tail Markovian processes was obtained by \cite{Godovanchuk-82}.  
For exponentially light jumps this question has been solved in 
one dimension by \cite{ImkPavlWetz-09, ImkPavWet10}. 
In addition, there is a large deviations theory for a special class of 
parameter dependent \textit{accelerated} noise with exponentially light tails 
by \cite{BudhirajaDupuisMaroulas-11} based on a variational representation. 
All these approaches yield exponential exit rates on the precise noise dependence. 
Further recent results on the first exit and metastability of L\'evy driven systems in finite and infinite dimensions 
were obtained by
\cite{ImkellerP-06,Pavlyukevich11,DHI13,HoePav-13}. 
It is worth mentioning that in the case of an overdamped particle subject to $\e$-small 
$\alpha$-stable noise, $\alpha \in (0,2)$, the expected exit time behaves \textit{polynomially} 
 \begin{equation}\label{e:StableExit}
\E \tau^\e \sim \frac{\bar V}{\e^\alpha}, 
\end{equation}
whereas the constant $\bar V$ has a very different interpretation. 
It is not the lowest \textit{height} of any mountain pass the continuous Brownian diffusion 
path has to \textit{climb} between different potential wells. 
Instead, it quantifies the \textit{tunelling} effect
of the ``large'' jumps that instantaneously overcome the 
(horizontal) \textit{distance} 
between the deterministic stable state 
where the process lingers most of the time between such ``large'' jumps  
and the exterior of its domain of attraction. 

In this article, we study such small heavy tailed perturbation 
of a beforehand non-Markovian dynamical system given as a linear delay differential equation. 
The simplest deterministic qualitative model of this kind is given by a linear retarded equation  
for the El Ni\~no-Southern Oscillation phenomenon (ENSO) in \cite{BatHir89}:
\ba\label{e:retarded}
\frac{\di}{\di t} X(t)&=AX(t)+ B X(t-r), 
\ea
where $A$ is the sum of all processes that induce local changes in the SST, 
that is, the horizontal advection, thermal damping, mean and anomalous upwellings on the 
vertical temperature gradients. The coefficient $B$ subsumes the effects 
of the equatorial Kelvin waves. 
Positive $A$ means that the sum of effects of upwelling and thermal advection 
dominate the thermal damping, so that the temperature grows. 
However, negative values of $B$ can induce stable or periodic solutions. 
For instance, the parameter choice $A=2.2$, $B=3.8$, and a delay time of $r=0.5$ as in \cite{BatHir89} 
leads to unstable oscillations with a period of approximately $3.0$ years and
a growth rate of $1.1$ while \cite{Burgers99} argued that changing this to $A=2.4$,
$B=-2.8$, and a delay time of $r=0.3$ leads to a period of approximately $4$ years and
a decay rate $1.5$.

A more complex non-linear double-well model with an additional cubic term $-X^3$ was considered by \cite{SuaSch1988}, 
followed by a number of papers
by e.g.\ \cite{MueCanZeb91,TziStoCanJar94,GhiZalTho08,ZalGhi10}.  

\cite{zabczyk1987exit} studied the following delay equation perturbed by a small Brownian motion $\e W$
\ba
\di X^\e(t)&= \mathcal{A}(X_t^\e)\,\di t+ \e\, \di W(t)\\
\ea
with a nonlinear Lipschitz continuous vector field 
$\mathcal{A}$. 
Applying a control theoretic approach to this equation 
he established a large deviations principle and 
showed that in analogy to the non-delay case discussed above 
the asymptotocs \eqref{e:GaussExit} holds true, however, with $\bar V$ being an abstract solution 
of a difficult delay control problem. 
Recently \cite{Lipshutz18} extended these results to 
the small noise SDDEs with multiplicative noise  
in the spirit of \cite{FreidlinW-12} and established the asymptotics of the 
first exit time of the type \eqref{e:GaussExit}.
\cite{azencott2018large} considered the retarded Gaussian delay equation 
as a Gaussian process and established the respective large deviations principle 
and the optimal exit paths with the help of the very elaborate Gaussian process theory.
In \cite{bao2016asymptotic} the authors study delay systems perturbed by small accelerated 
L\'evy noise with light tails in the spirit of \cite{BudhirajaDupuisMaroulas-11}. 
More on stochastic double-well systems can be found in \cite{Masoller-02,Masoller-03} and 
for L\'evy noise also in \cite{HuangTaoXu-11}, where the authors study 
the asymptotics in the limit of small delay. 
 
In this paper we study the first exit problem from an interval 
of the delay differential equation 
\ba
\di X^\e(t)&= \mathcal{A} X^\e_t\,\di t+ \e F(X^\e_t)\, \di Z(t),\\
\ea
with a general \textit{linear} stable finite delay $\mathcal{A}$ 
perturbed by a small multiplicative heavy-tailed L\'evy noise $\e Z$, 
including $\alpha$-stable but also more general weakly tempered perturbations.  
The phenomenological reason for our setting is that on the one hand 
we recover the rate \eqref{e:StableExit}, however, 
we detect a new \textit{non-normal growth effect} in the factor $\bar V$, 
which we can calculate explicitly in the case of a retarded system \eqref{e:retarded}. 
This effect accounts for the non-zero probability of small jump increments, 
which leads to an exit due to deterministic motion well after the occurrence of this jump 
and can be seen in the asymptotic distribution of the exit location, 
which in contrast to the non-delay case exhibits a point measure precisely 
on the boundary of the exit interval. 
It is easily seen that this effect vanishes if we send the memory depth to $0$. 
In other words, at first sight these results for non-Markovian systems appear surprising, 
however, since the delay time $r$ is negligible w.r.t.\ the exit time scale 
$\e^{-\alpha}$, the system behaviour is ``almost'' Markovian.
Nevertheless the memory affects the prefactor in the asymptotics 
of the first exit time and the limiting 
distribution of the exit location.

The methodological reason for considering this equation is 
the adaption of the proof strategy by \cite{Godovanchuk-82} and 
\cite{ImkPavlWetz-09} which is an elementary but very helpful
application the Markov property which
seems to be suitable for adaptation in different contexts of the physics literature. 
In addition, our setup covers generic non-degenerate potential gradient systems the linearized around their stable state. 

A technical reason to study this particular problem is that 
in order to trace this effect we need a precise 
understanding of the deterministic dynamics.  
In particular, several important properties of this equation 
are readily given in the literature, 
such as existence, uniqueness of solutions and the invariant measure in \cite{GK-00} and  
the \textit{segment Markov} property in \cite{RRvG06}.

The article is organized as follows. After the general setup we present our main result 
in Theorem 2.2 followed by the discussion and examples, 
where we compare our results to the linear setting of \cite{zabczyk1987exit} 
and explicitly calculate the nonlinear growth factor. 
The rigorous proofs are postponed to the Mathematical Appendix and consist of two parts. In Section \ref{s:general} we give general estimates on the deterministic 
relaxation dynamics and then on the stochastic perturbation. 
In Section \ref{s:markov} in  we show a generic upper and a lower bound 
on the segment Markov process reducing the dynamics 
to four scenarios on a finite interval. 
This finite interval dynamics 
is treated in Section \ref{s:Main estimates} for each of 
the cases what leads to the proof of the main result.

\section{Object of Study and Results}

\subsection{General setup}

The model under consideration is the following linear delay equation 
with finite memory $r>0$ perturbed by small L\'evy noise defined below.
For $r>0$ fixed and each $T>0$, we denote by $D[-r,T]$ 
the space of real valued right-continuous functions $\phi\colon [-r, T] \to \bR$ with left limits (the so-called c\`adl\`ag 
functions). Analogously we define the space $D[-r,\infty)$. For a function $\phi\in D[-r,\infty)$ we 
we introduce the segment
of $\phi$ at time $t \geq 0$ as the function $\phi_t(\cdot)\in D[-r,0]$ defined by
$\phi_t(s)=\phi(t+s)$ for $s\in[-r,0]$.
For a function $\phi \in D[-r,0]$, we denote its uniform norm by $\|\phi\|_r:=\sup_{t\in[-r,0]}|\phi(t)|$.

Let $\mu$ be a finite, signed measure on the interval $[-r, 0]$, so-called the \textit{memory measure}. 
Consider the following underlying deterministic linear delay equation
\ba
\label{e:de}
x(t; \phi) &= \phi(0) + \int_0^t \int_{[-r,0]} x(s+u; \phi)\, \mu(\di u)\, \di s,\quad t\geq  0,\\ 
x(t; \phi) &= \phi(t),\quad  t\in [-r,0),
\ea
where $\phi\in D[-r, 0]$. It is well known that this equation has a unique solution 
which e.g.\ can be obtained by the method of steps as in \cite{Hale-Lunel-93}.
Similarly to the case of linear ODEs or PDEs, 
the solution $x(\cdot; \phi)$ can be written down explicitly 
as a convolution integral, namely,
\ba
\label{e:conv}
x(t;\phi) &= \phi(0) x^*(t) + \int_{[-r, 0]} \int_u^0  x^*(t-s+u)\phi(s)\, \di s \, \mu(\di u) ,\quad t\geq 0 \\ 
x(t;\phi) &= \phi(t) ,\quad  t\in [-r, 0),  
\ea
where the fundamental solution $x^*(\cdot)$ is the unique solition of 
\eqref{e:de} with the initial segment $\phi^*(t)=0$, $t\in[-r,0)$, and $\phi^*(0)=1$.
%

Let $(\Omega, \rF, (\rF_t)_{t\geq 0}, \P)$ be a stochastic basis satisfying the usual 
conditions in the sense of \cite{Protter-04} and let $Z=(Z(t))_{t\geq 0}$ be an adapted real-valued 
L\'evy process
with the characteristic triplet $(\sigma^2, d, \nu)$. The marginal laws of $Z$ are described by the L\'evy--Khintchine formula
\ba
\label{e:LKZ}
\ln \E \ex^{\i u Z(t)}=-\frac{\sigma^2}{2}u^2t +\i\, d u t+t\int\Big(\ex^{\i uz}-1-\i u z\bI_{[-1,1]}(z)\Big)\,\nu(\di z),
\ea
where the Gaussian variance $\sigma^2\geq 0$, the drift $d\in\bR$, and the L\'evy jump 
measure $\nu$ satisfies $\nu(\{0\})=0$ and $\int (z^2\wedge 1)\,\nu(\di z)<\infty$. 
For the general theory on L\'evy processes see \cite{Sato-99,Applebaum-09}.
To introduce multiplicative noise into the equation \eqref{e:de}, 
we define the ``diffusion'' coefficient $F\colon D[-r,0]\to D[-r,0],$ which we 
assume to be functional Lipschitz, i.e.\
there is a constant $L>0$ such that for all $\phi$, $\psi \in D[-r,0]$ we have
\begin{equation}
\label{eq:functionalLipschitz}
\|F(\phi)  -  F(\psi) \|_r \leq L \|\phi-\psi\|_r.
\end{equation}
The functional $F$ can be for example of the form $F(\phi)(t)=f(\phi(t-r_1),\dots,\phi(t-r_n)))$
for point delays $r_i\in [0,r]$, and $f$ being Lipschitz in all its arguments, such that
\ba
F(X^\e_t)(0-)=f\Big(X^\e(t-r_1-),\dots,X^\e(t-r_n-)\Big).
\ea
Further examples of $F$ can be found in \cite[Example 2.1]{RRvG06}.

Under the assumptions formulated above we consider the stochastic 
delay differential equation with an initial condition $\phi \in D[-r,0]$ and $\e>0$
\begin{align}
\label{e:SDDE}
X^\e(t) &= 
\phi(0) + \int_0^t \Big[\int_{-u}^0 X^\e(s+u) \mu(\di u)\Big]\, \di s 
+ \e \int_0^t F(X^\e_s)(0-) \, \di Z(s), \quad  t\geq 0,  \\
X^\e(t) &= \phi(t), \quad t\in [-r, 0).                
 \end{align}
and denote $X^\e(\cdot;\phi)=X^\e(\cdot)$ its solution. 

\noindent The multiplicative noise term is understood as the It\^o stochastic integral which requires 
the predictability of the integrand $F(X^\e_s)(0-)=\lim_{h \uparrow 0} F(X^\e_s)(h)$.
Note that in the pure Gaussian continuous setting $F(X^\e_s)(0-)=F(X^\e_s)(0)$.

\begin{thm}[\cite{RRvG06}] 
Fix $\e\in (0,1]$ and let $F$ be functional Lipschitz. 
Then for any $\phi \in D[-r,0]$ there exists a unique solution $X^\e$ of equation \eqref{e:SDDE} which satisfies the convolution formula 
\begin{equation}\label{eq: convolution formula}
X^\e(t; \phi) =  \phi(0)x^*(t)+\int_{[-r,0]}\int_s^0 x^*(t+s-u)\phi(u)\,\di s\, \mu(\di u) + \e\int_0^t x^*(t-s) F(X^\e(\phi))(s-)\, \di Z(s), t\geq 0. 	 
\end{equation}
The solution $X^\e$ is Markovian in the segment space $D[-r,0]$, i.e.\ for each $0\leq s\leq t$ and each measurable set 
$B\subseteq D[-r,0]$
\ba
\label{e:MP}
\P(X^\e_t\in B|\rF_s)=\P(X^\e_t\in B|X_s^\e)\text{ a.s.}
\ea
\end{thm}

\subsection{Main results}

Our main result characterizes the interplay between the deterministic stability and the power laws 
of the noise. We will need the following Hypotheses.

\smallskip

\noindent
$\textbf{H}_\mu:$
We assume that the delay equation \eqref{e:de} is stable, i.e.\
the memory measure $\mu$ satisfies 
\begin{equation}
\label{e:neg-char-exp}
-\Lambda:= \sup_{\la_\iota} \operatorname{Re} \la_\iota < 0, 
\end{equation}
for $\la_\iota$ being a solution of the characteristic equation
\ba
\la - \int_{-r}^0 \ex^{ u \la}\, \mu(\di u) = 0.
\ea
Condition \eqref{e:neg-char-exp} implies that 
for each $\lambda<\Lambda$ there is a constant $K=K(\lambda)>0$ such that
\ba
\label{e:K}
|x^*(t)|\leq K\ex^{-\lambda t},\quad t\geq 0.
\ea
Zero is a stable state. For any function $\phi\in D[-r,0]$ such that $\phi(t)=0$, $t\in[-r,0)$,
we assume that $F(\phi)(0-)=F_0\neq 0$.

\smallskip

\noindent
$\textbf{H}_\nu:$ The goal of this paper is to treat the heavy tail phenomena. 
A convenient analytic tool for this is the theory  
of regularly varying functions, i.e.\ functions which behave asymptotically like power
functions.
Let $\lambda_\e$ denote the tail of the L\'evy measure $\nu$,
\ba
\lambda_\e=\int_{|z|>\frac{1}{\e}}\nu(\di z),\quad \e>0.
\ea
We assume that there exist $\alpha>0$ and a non-trivial self-similar Radon measure $\bar\nu$ on $\bar \bR\backslash\{0\}$
such that
for any $u>0$ and any Borel set $A$ bounded away from the origin, $0\notin \bar{B}$, 
the following limit holds true:
\begin{equation}
\label{e:rv}
\bar\nu(u B)=\lim_{\e\to 0} \frac{\nu(uB/\e)}{\lambda_\e} = 
\frac{1}{u^\alpha} \lim_{\e\to 0} \frac{\nu(B/\e )}{\lambda_\e}= \frac{1}{u^\alpha} \bar\nu(B).
\end{equation}
In particular, there exists a non-negative function $l$ slowly varying at zero 
such that 
\ba
\lambda_\e = \e^{\alpha} l(\e)\quad\mbox{ for all} \quad \e>0.
\ea
The self-similarity property of the limiting measure $\bar \nu$ 
implies that is has no atoms, $\bar\nu(\{z\})=0$, $z\neq 0$, and hence, in the one-dimensional case, $\bar\nu$ always has 
the power density
\ba
\label{e:barnu}
&\bar\nu(\di z)=\bar c_-\frac{\bI(z<0)}{|z|^{1+\alpha}}\,\di z+\bar c_+\frac{\bI(z>0)}{z^{1+\alpha}}\,\di z,\quad
\alpha>0,\quad \bar c_\pm\geq 0, \ \bar c_-+\bar c_+>0.
\ea

For the interval $[a,b]$, $a<0<b$, 
we define the first exit time
\ba
\tau^\e=\tau^\e(\phi)&=\inf\{t\geq 0\colon X^\e(t;\phi)\notin [a,b]\}.
\ea

\noindent Due to the continuity of the fundamental solution we obtain that the set of jump sizes
\ba
{}[e_-,e_+]=\{z\in\bR    \text{ such that } z\cdot F_0\cdot x^*(t)\in [a,b] \text{ for all }t\geq 0  \}
\ea
is a closed interval with $e_-<0<e_+$.
Denote by 
\ba
\label{e:E}
E&=E(a,b)=  E(a,b;A,B,r):=\{z\in\bR\colon \exists \, t\geq 0\text{ such that } z\cdot F_0\cdot x^*(t)\notin [a,b]  \}\\
&=[e_-,e_+]^c=(-\infty,e_-)\cup (e_+,+\infty)\\
\ea
the set of jump sizes which cause the exit from the interval $[a,b]$. Furthermore consider the sets
\ba
&E_b=\{z\in\bR\colon  z\cdot F_0\cdot x^*(t)\text{ exits from $[a,b]$ into $(b,\infty)$} \},\\
&E_a=\{z\in\bR\colon  z\cdot F_0\cdot x^*(t)\text{ exits from $[a,b]$ into $(-\infty,a)$} \},\\
&E=E_a\bigsqcup E_b,\\
&E_b^j(v)=\{z\in\bR\colon  z\cdot F_0\in (v,\infty) \},\quad v> b,\\
&E_a^j(v)=\{z\in\bR\colon  z\cdot F_0\in (-\infty,v) \},\quad v< a.
\ea
Recall that the homogeneity of the measure $\bar\nu$ guarantees that $\bar\nu(\{e_\pm\})=0$. 

\begin{thm}
\label{t:main}
Let Hypotheses \emph{$\textbf{H}_\mu$} and \emph{$\textbf{H}_\nu$} hold true. 
Let $[a,b]$ be an interval, $a<0<b$, and let $\phi\in D[-r,0]$ be an initial segment with no exit, i.e.\
such that 
\ba
\label{e:phinoexit}
a< \inf_{t\in[-r,\infty) }x(t,\phi)\leq \sup_{t\in[-r,\infty)}x(t,\phi)< b.
\ea
For the set $E$ defined in \eqref{e:E},
assume that $\bar\nu(E)>0$. 

\noindent  
1. For each $u>0$ we have 
\ba
\lim_{\e\to 0}\P_\phi(\lambda_\e \tau^\e>u)= \ex^{-u\bar\nu(E)}.  
\ea
2. In addition, we have 
\ba
\lim_{\e\to 0} \lambda_\e \E_\phi \tau^\e=\frac{1}{\bar\nu(E)}. 
\ea
3. In the limit $\e\to 0$, the exit location is given by 
\ba
X^\e(\tau^\e;\phi)\stackrel{\di}{\longrightarrow} 
\Pi_a^j\cdot \frac{\bar c_-\bI_{(-\infty,a)}(z)}{|z|^{1+\alpha}}\,\di z
+
\Pi_a^c \cdot \delta_a(\di z)
+\Pi_b^c \cdot \delta_b(\di z)
+\Pi_b^j\cdot \frac{\bar c_+\bI_{(b,\infty)}(z)}{z^{1+\alpha}}\,\di z,
\ea
where
\ba
\Pi_b^j&=\lim_{v\downarrow b}\frac{\bar\nu(E_b\backslash E_b^j(v))}{\bar\nu(E)},\qquad 
\Pi_a^j=\lim_{v\uparrow a}\frac{\bar\nu(E_a\backslash E_a^j(v))}{\bar\nu(E)},\\
\Pi_b^c&=\Pi_b-\Pi_b^j,\qquad
\Pi_a^c=\Pi_a-\Pi_a^j.
\ea
Note that
\ba
\Pi_a^j+\Pi_a^c+\Pi_b^c+\Pi_b^j=1.
\ea
\end{thm}
We discover the positive weight $\Pi^c = \Pi_a^c + \Pi_b^c$ on the boundary $\{a, b\}$ 
which represents the probability of an asymptotically continuous exit from the interval $[a, b]$ 
and stemps from the non-normal growth effect of the deterministic delay equation.

\subsection{Examples and Discussion}

We start this section with examples of L\'evy processes with regularly varying heavy tails which satisfy Hypothesis $\textbf{H}_\nu$.

\begin{exa}
Any $\alpha$-stable L\'evy process with the stability index $\alpha\in (0,2)$, the skewness parameter $\beta\in[-1,1]$, and the scale
parameter $c>0$ satisfies Hypothesis $\textbf{H}_\nu$.  
Indeed, such a L\'evy process $Z$ has the characteristic function
\ba
\E \ex^{\i u Z(1)}=\begin{cases}                
	  \exp\Big( -c|u|^\alpha (1-\i \beta \tan\frac{\pi\alpha}{2})\sgn u  \Big),\quad \alpha\in (0,1)\cup (1,2),\\
	   \exp\Big( -c|u| (1+\i \beta \frac{2}{\pi}\sign u \ln |u|  \Big),\quad \alpha=1,
                   \end{cases}
\ea
and (see \cite[Chapter 3.5]{UchaikinZ-99}) its jump measure $\nu$ has the form
\ba
&\nu(\di z)=\Big(c_-\frac{\bI(z<0)}{|z|^{1+\alpha}}+c_+\frac{\bI(z>0)}{z^{1+\alpha}}\Big)\,\di z\\
\ea
with
\ba
c_\pm=\begin{cases}
\displaystyle \frac{c\cdot(1\pm\beta)}{2|\Gamma(-\alpha)|\cos(\frac{\alpha\pi}{2})},\quad \alpha\in (0,1)\cup (1,2),\\
\displaystyle \frac{c\cdot(1\pm\beta)}{\pi},\quad \alpha=1.
\end{cases}
\ea
In this case, the limiting measure $\bar \nu$ coincides with $\nu$, so that $\bar c_\pm=c_\pm$ in \eqref{e:barnu} and 
\ba
\label{e:lambdaeps}
\lambda_\e=\frac{c_++c_-}{\alpha}\e^\alpha .
\ea
\end{exa}
\begin{exa}
Weakly tempered stable L\'evy processes form another important class of perturbations with heavy tails. Various ways of tempering 
have been introduced, e.g.\ in \cite{SokolovCK-04,Rosinski-07}. Roughly speaking, small jumps of a weakly tempered 
$\alpha$-stable L\'evy process
look like those of an $\alpha$-stable process, but the large jumps, and hence the tails of the p.f.d.\ are of the order $|x|^{-1-r}$
for some $r>0$. It is easy to construct a weakly tempered stable L\'evy process with the help of its jump measure defined as
\ba
\nu(\di z)
&=\Big(c_-\frac{\bI(z<0)}{|z|^{1+\alpha_1}(1+z^2)^{\alpha_2/2}}+c_+\frac{\bI(z>0)}{z^{1+\alpha_1}(1+z^2)^{\alpha_2/2}}\Big)\,\di z,\\
&\alpha_1\in(0,2),\ \alpha_2>0,\quad c_\pm\geq 0,\quad c_++c_->0.
\ea
In this case the limiting measure is
\ba
&\bar\nu(\di z)= \Big(c_-\frac{\bI(z<0)}{|z|^{1+\alpha}}+c_+\frac{\bI(z>0)}{z^{1+\alpha}}\Big)\,\di z,\quad \alpha=\alpha_1+\alpha_2>0,
\ea
and $\lambda_\e$ is as in \eqref{e:lambdaeps}.
\end{exa}
\begin{exa}
The L\'evy measures from the previous examples can be ``contaminated'' by some slowly varying function $l_\pm=l_\pm(z)$, 
like $l(z)=\ln(1+|z|)$ or
any finite nonnegative function $l$ such that there exists limits $l_\pm =\lim_{z\to\pm\infty} l_\pm(z)\in (0,\infty)$, e.g.\
one can consider jump measures of the form
\ba
&\nu(\di z)=\Big(l_-(z)\frac{\bI(z<0)}{|z|^{1+\alpha}}+l_+(z)\frac{\bI(z>0)}{z^{1+\alpha}}\Big)\,\di z.
\ea
Moreover, the additional influence of any drift $d$ and a Brownian motion $\sigma W$ (see \eqref{e:LKZ}) 
is negligible in comparison to the heavy jumps and 
does not change the asymptotic characteristics of the exit time and location.
\end{exa}

\medskip 
\noindent The sets $E$, $E_a$, $E_b$, $E_b^j(v)$, and $E_b^j(v)$ appearing in Theorem \ref{t:main} 
are determined in terms of the characteristics of the fundamental solution 
$x^*$. Generally, the fundamental solution is not known explicitly, however its maximum and minimum can be obtained numerically.
By the definition of the fundamental solution, its maximum satisfies
\ba
M=\max_{t\geq 0} x^*(t)\geq 1,
\ea
is well defined and is attained somewhere on $[0,\infty)$.
On the other hand the minimum of $x^*(\cdot)$ may not be attained (e.g.\ for $x^*(t)=\ex^{-t}$) and we set
\ba
m=\inf_{t\geq 0}x^*(t)\leq 0.
\ea
Assume for definiteness that $F_0>0$. Then the values $e_\pm$ 
and the set $E=[e_-,e_+]^c$ can be calculated explicitly as  
\ba
\label{e:epem}
&e_+=\sup\{z>0 \text{ such that } z\cdot F_0\cdot x^*(t)\in (a,b) \text{ for all }t\geq 0\}
=\frac{|a|}{F_0|m|} \wedge \frac{b}{F_0M},\\
&e_-=\inf\{z<0 \text{ such that } z\cdot F_0\cdot x^*(t)\in (a,b) \text{ for all }t\geq 0\}
=-\Big(\frac{|a|}{F_0 M}\wedge\frac{b}{F_0|m|}\Big),
\ea
where we set $\frac10=+\infty$. Analogously one can determine the sets $E_a$, $E_b$, $E_a^j(v)$, and $E_b^j(v)$
but the explicit formulae cannot be given here in general since one needs additional information
whether $M$ or $m$ is attained first.

We finish the discussion with the analysis of linear retarded equation. 

\begin{exa}\label{ex:Levy}
Consider the linear retarded equation
\ba
\dot X^\e(t)&=A X^\e(t)+B X^\e(t-r)+ \e \dot Z(t),\quad t\geq 0,\\ 
X^\e(t)&=\phi(t),\ t\in[-r,0], 
\ea
driven by a symmetric additive $\alpha$-stable noise $Z$
with the L\'evy measure $\nu(\di z)=|z|^{-1-\alpha}\,\di z$, $\alpha\in(0,2)$.
The stability region of the deterministic equation $\dot x(t)=A x(t)+B x(t-r)$ obviously 
coincides with the stability region of the rescaled equation $\dot y=\tilde A y(t)+\tilde B y(t-1)$, 
where $\tilde A=Ar$, and $\tilde B=Br$; their fundamental solutions satisfy $x^*(r t)=y^*(t)$, $t\geq -1$.
The stability region of the parameters $(\tilde A,\tilde B)$ 
is depicted on Fig.~\ref{f:stability}. It is bounded by the upper straight
line $\tilde A+\tilde B=0$, $\tilde A<1$, and the lower line
which is given parametrically as
\ba
\tilde A=\zeta\cot\zeta,\quad \tilde B=-\frac{\zeta}{\sin\zeta},\quad \zeta\in(0,\pi),
\ea
see \cite[Section 5.2 and Theorem A.5]{Hale-Lunel-93} for more details.
\begin{figure}
\centerline{\includegraphics{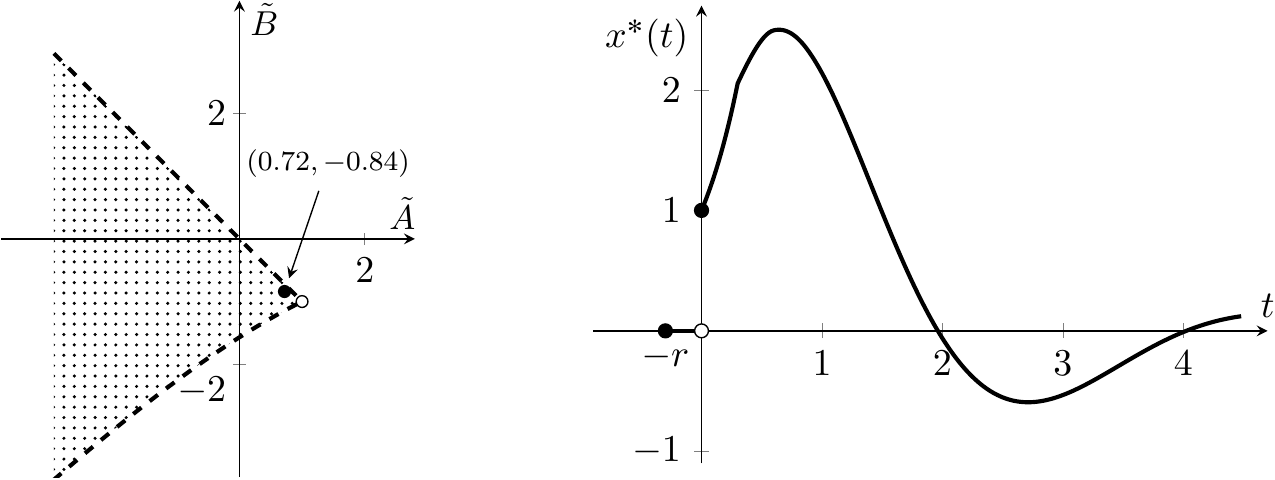}}
\caption{Left: Stability region of the linear retarded equation $\dot y(t)=\tilde A y(t)+\tilde B y(t-1)$. The point 
$(0.72,-0.84)$ corresponds to the parameter values $A=2.4$, $B=-2.8$, $r=0.3$ in the example from \cite{Burgers99}. 
Right: The fundamental solution $t\mapsto x^\ast(t)$ of the equation $\dot x(t)=2.4 x(t)-2.8 x(t-0.3)$
\label{f:stability}}
\end{figure}
\begin{figure}
\centerline{\includegraphics{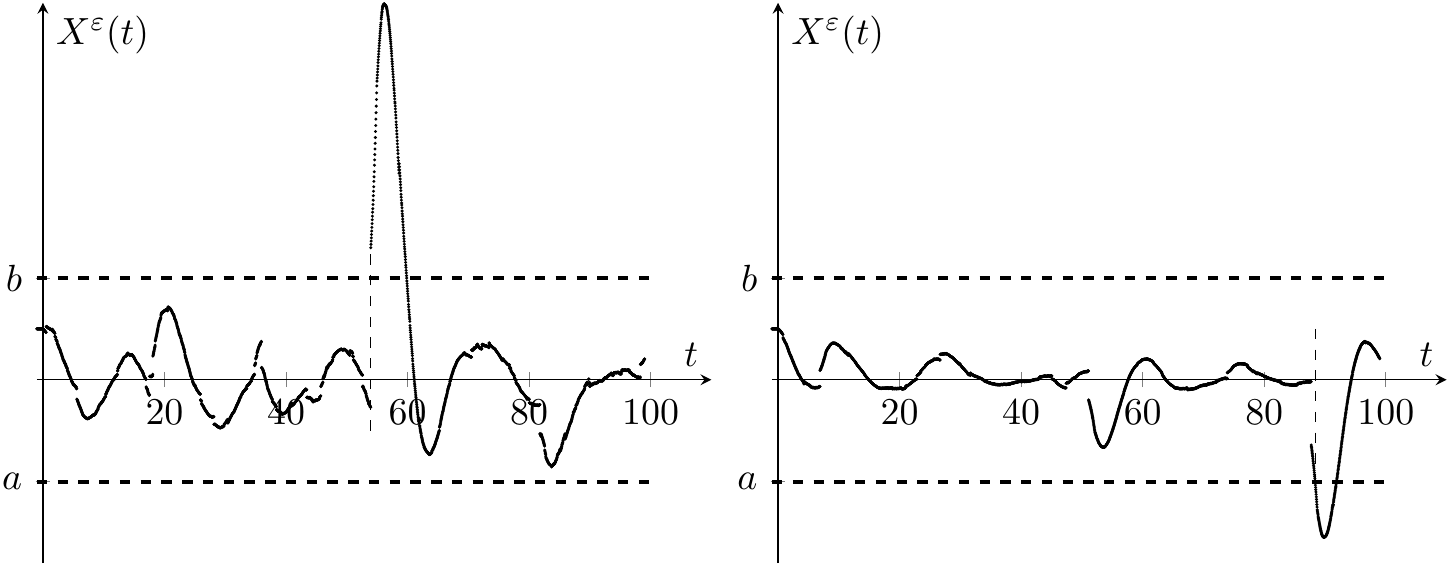}}
 \caption{The most probable exit patterns of the perturbed delay equation $\dot X^\e(t)=2.4 X^\e(t)-2.8 X^\e(t-0.3)+\e \dot Z(t)$ from
\cite{Burgers99}.
Left: The instant exit due to a large jump. Right: The exit due to a large jump and non-normal growth. In this case, the limiting
exit location $X^\e(\tau^\e)$ is supported by the end points $a$ and $b$.  \label{f:paths}}
\end{figure}
\begin{figure}
\centerline{\includegraphics{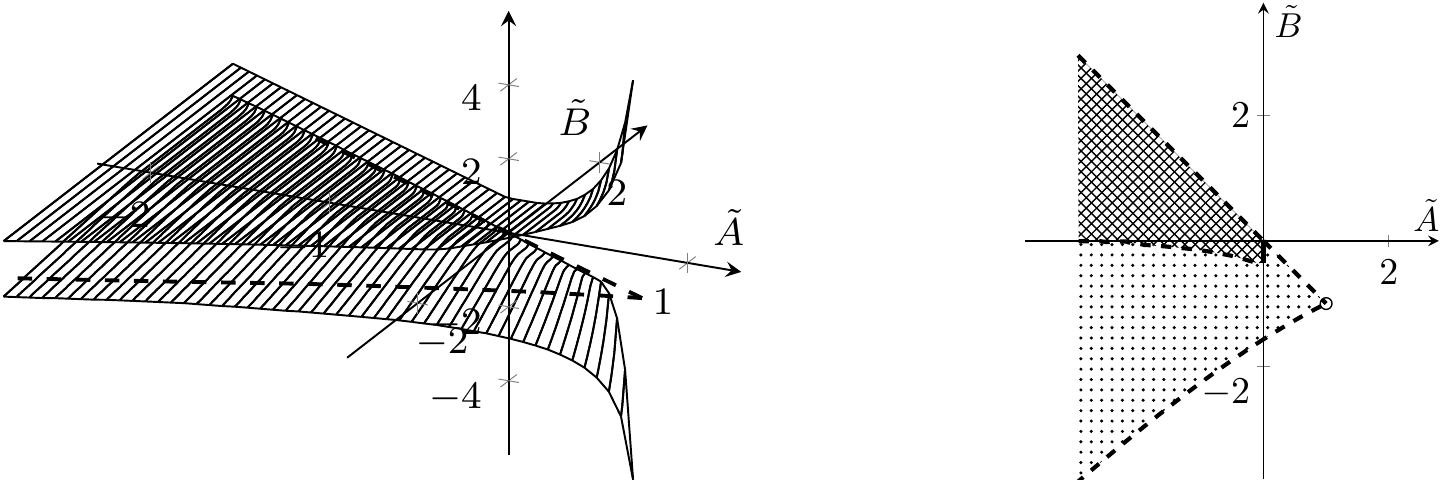}}
\caption{Left: The maximum $M$ and the infimum $m$ of the fundamental solution $y^*$ of the retarded delay equation 
$\dot y(t)=\tilde A y(t)+\tilde B y(t-1)$ for $(\tilde A,\tilde B)$ in the stability region.
Right: The domain of parameters $(\tilde A,\tilde B)$ where $M=1$ and $m=0$. For these parameters, 
the asymptotics of the mean exit times \eqref{e:exitdelay} and \eqref{e:nodelay} of the equation with and without delay coincide. 
\label{f:domain}}
\end{figure}

\noindent The most probable exit trajectories of the perturbed delay equation $\dot X^\e(t)=2.4 X^\e(t)-2.8 X^\e(t-0.3)+\e \dot Z(t)$ from
\cite{Burgers99} are presented on Fig.\ \ref{f:paths}. Taking into account the formulae \eqref{e:lambdaeps} and 
\eqref{e:epem} we find that the mean first exit time from an interval $[a,b]$ around zero satisfies 
\ba
\label{e:exitdelay}
\E_\phi \tau^\e\approx\frac{1}{2\e^\alpha}\Big( 
\frac{M^\alpha}{|a|^\alpha}\vee\frac{|m|^\alpha}{b^\alpha}
+
\frac{|m|^\alpha}{|a|^\alpha} \vee \frac{M^\alpha}{b^\alpha}
\Big)^{-1}
\ea
where the values $M=M(A,B,r)$ and $m=m(A,B,r)$ depend on the values $A$, $B$, and $r$
in a complex nonlinear way, see Fig.\ \ref{f:domain}. 
Denoting by
$\tilde M=\tilde M(\tilde A,\tilde B)=M(Ar,Br,1)$ and $\tilde m=\tilde m(\tilde A,\tilde B)=m(Ar,Br,1)$ 
the extreme values of the fundamental solution $y^*$ we get that
\ba
\label{e:exitdelaytilde}
\E_\phi \tau^\e\approx\frac{1}{2\e^\alpha}\Big( 
\frac{\tilde M^\alpha}{|a|^\alpha}\vee\frac{|\tilde m|^\alpha}{b^\alpha}
+
\frac{|\tilde m|^\alpha}{|a|^\alpha} \vee \frac{\tilde M^\alpha}{b^\alpha}
\Big)^{-1}.
\ea 
The equation without delay, i.e.\ for $B=0$, is stable for $A<0$ with the 
fundamental solution $x^*(t)=\ex^{A t}$, so that $M=1$ and $m=0$, and mean exit time has the asymptotics
\ba
\label{e:nodelay}
\E_\phi \tau^\e\approx\frac{1}{2\e^\alpha}\Big(\frac{1}{|a|^\alpha}+\frac{1}{b^\alpha}\Big)^{-1},
\ea
see \cite{Godovanchuk-82,ImkellerP-06,ImkellerP-06a}. It is interesting to note that $M=1$ and $m=0$ holds for parameters $\tilde A=Ar$
and $\tilde B=Br$ from a larger domain. This domain can be determined numerically and is depicted on  Fig.\ \ref{f:domain}.
For parameters in this domain, the asymptotics of the mean exit times \eqref{e:exitdelay} and 
\eqref{e:nodelay} of the equation with and without delay coincide, and the asymptotic exit location has no atoms in $a$ and $b$.
Hence, the first exit dynamics of the delay equation is effectively the same as for the equation without delay.
\end{exa}

\noindent Eventually it is instructive to compare our results with the asymptotics of the exit time in the Gaussian case which 
has been studied for the first time by \cite{zabczyk1987exit}.

\begin{exa}[Example 1, \cite{zabczyk1987exit}]\label{ex:Zabc}
Consider a linear retarded equation driven by the Brownian additive noise
\ba
\dot X^\e(t)&=A X^\e(t)+B X^\e(t-r)+ \e \dot W(t),\quad t\geq 0,\\ 
X^\e(t)&=\phi(t),\ t\in[-r,0], 
\ea
The stability region of this equation has been described in the previous example. With the help of the 
large deviations theory one obtains that
for any $\phi$ with no exit (see \eqref{e:phinoexit})
\ba
\label{e:exitGauss}
\e^2 \ln \E_\phi \tau^\e\sim \frac{(a\wedge b)^2}{r G(\tilde A,\tilde B)},
\ea
where the value $G$ is obtained in terms of the fundamental solution $y^\ast$ of $\dot y(t)=\tilde Ay(t)+\tilde By(t-1)$:
\ba
\label{e:G}
G(\tilde A,\tilde B)=2\int_0^\infty (y^*(t))^2\,\di t&=\frac{2}{\pi}\int_0^\infty \frac{\di t}{(\tilde A+\tilde B\cos t)^2+(t+\tilde B\sin t)^2}\\
&=\begin{cases}
\displaystyle \frac{\tilde B q^{-1}\sin q-1}{\tilde A+\tilde B\cos q} ,\quad  \tilde B+ |\tilde A|<0,\ q=\sqrt{\tilde B^2-\tilde A^2},\\
\displaystyle  \frac{1+|\tilde A|}{2|\tilde A|}, \quad \tilde B=\tilde A<0,\\
\displaystyle  \frac{\tilde Bq^{-1}\sinh q -1}{\tilde A+\tilde B\cosh q} , \quad  \tilde A+|\tilde B|<0,\ q=\sqrt{\tilde A^2-\tilde B^2},
\end{cases}
\ea
see \cite[Lemma 1.5]{Shaikhet-11}. For the equation without delay, i.e.\ with $B=\tilde B=0$ one gets from 
\eqref{e:exitGauss} and \eqref{e:G} the well known exponentially large Kramers' time
\ba
\E_\phi \tau^\e\sim \ex^{|A|\frac{(a\wedge b)^2}{\e^2}},\quad A<0.
\ea
\end{exa}

\bigskip

\section{Conclusion} 

In this article we solve the first exit time and location 
problem from an interval $[a, b]$, $a < 0 < b$,  
in a general class of stable linear delay differential equations 
for $\e$-small multiplicative, power law noise, such as $\alpha$-stable L\'evy flights. 
In particular, we cover the linearization of gradient systems 
close to a stable state. 

We recover, on the one hand, the asymptotic 
polynomial exit rate of the order $\varepsilon^{-\alpha}$ 
known in different Markovian settings 
which comes with an $\varepsilon$-independent prefactor.  
In the delay case this prefactor depends nonlinearly on the memory depth $r>0$ 
and can significantly reduce the expected exit times 
compared to the non-delay case 
reflecting a non-normal growth phenomena 
of the deterministic delay dynamics 
which changes the stochastic dynamics. 
This mirrors the situation for Brownian perturbations, 
where such an effect has been known since a long time 
as explained in Example \ref{ex:Zabc}.
The case of retarded equations is explained 
in detail in Example \ref{ex:Levy}. 

Secondly, this non-normal growth becomes evident 
in the limiting exit location from the interval $[a, b]$ 
where in contrast to the non-delay case with jump exits 
we detect a point mass on the boundary. 
This mass stems from random trajectories which exit 
essentially due to deterministic motion. 

The method of proof applied in the mathematical 
appendix consists of a series of elementary 
but new estimates valid for general Markov processes 
and well suited to be adapted to other settings.

\bigskip 

\section{Mathematical Appendix: Proof of the main Theorem \ref{t:main}\label{s:appendix}} 

\subsection{General estimates\label{s:general}}

This section reduces the first exit problem 
in three consecutive steps
from a global problem to several 
smaller problems all of which are local in nature. 
We start by showing that 
the perturbed dynamics is 
dominated by the deterministic dynamics 
plus the error term of the perturbation. 
This is carried out for rather general perturbations 
by stochastic processes given as semimartingales, 
which include the stochastic integrals 
$\e \int_0^t F(X^\e_s)(0-) \, \di Z(s)$ under consideration in \eqref{e:SDDE}. 
In the sequel we show that the solution driven only by bounded jumps, 
that is, before a first large jump happens  
remains close to the deterministic solution. 
Finally we establish upper and lower bounds 
for the distribution tails and the expectation 
of the exit times (including exit locations) 
for general Markov processes provided we have enough control over 
the short term behavior, which is left for Section~\ref{s:Main estimates} 
in order to conclude.

\subsubsection{Estimates on the perturbed dynamics}

For generalities of semimartingales we refer to the book by \cite{Protter-04}. 
In general, semimartingales are the class of stochastic It\^o integral processes. 

\begin{lem}
\label{l:est}
Let $S$ be a c\`adl\`ag semimartingale, $S(0)=0$, $\phi\in D[-r,0]$ and
let $X$ be a stochastic process satisfying 
\ba
X(t,\phi)&=\phi(0)+\int_0^t\Big[\int_{[-r,0]} X(s+u,\phi)\,\mu(\di u) \Big]\,\di s + S(t), \qquad t\geq 0,\\
X(t,\phi)&=\phi(t),\quad t\in[-r,0).
\ea
Then for each $\lambda\in (0,\Lambda)$ there is $C=C(\lambda,\mu,r)>0$ such that for $t\geq 0$
\begin{align}
\label{e:est1}
&|X(t,\phi)| \leq C\cdot\Big(\|\phi\|_r\cdot \ex^{-\lambda t}+ \sup_{s\in[0,t]}|S(s)|\Big),\\
\label{e:est2}
&|X(t,\phi)-\phi(0)x^*(t)|\leq C \cdot\Big(\sup_{s\in[-r,0)}|\phi(s)| \cdot \ex^{-\lambda t} +   \sup_{s\in[0,t]}|S(s)|\Big),\\
\label{e:est3}
& |X(t,\phi)-x(t,\phi)|\leq C \cdot \sup_{s\in[0,t]}|S(s)|.
\end{align}
\end{lem}
\noindent Inequality \eqref{e:est1} estimates the growth of the perturbed solution $X$ in terms of the size of the 
initial segment and the perturbation $S$, whereas \eqref{e:est2} quantifies the memory effect in the initial segment and the 
noise. Inequality \eqref{e:est3} controls the deviation caused by the noise alone.
\begin{proof}
Let $\lambda\in (0,\Lambda)$ and the corresponding $K=K(\lambda)$ in estimate \eqref{e:K} be fixed.

\noindent
1. By linearity we note that the 
the difference $Y(t;\phi)=X(t;\phi)-X(t;0)$ satisfies the homogeneous deterministic delay equation \eqref{e:de}, that is, 
$Y(t;\phi)=x(t;\phi)$ for all $t\geq -r$, and hence the convolution formula \eqref{e:conv} immediately implies
\ba
\label{e:C1}
|Y(t)|&\leq K |\phi(0)|\cdot \ex^{-\lambda t}+ K \int_{[-r,0]}\Big[ \int_u^0\Big[ \ex^{-\lambda (t-s+u)}\cdot |\phi(s)|\,\di s\Big]\,|\mu(\di u)|\\
&\leq K\cdot  |\phi(0)|\cdot \ex^{-\lambda t}+ K\cdot \frac{(e^{\la r}-1)}{\la} \cdot \|\phi\|_r \cdot |\mu|[-r, 0] \cdot \ex^{-\lambda t}\\[2mm]
&\leq C_1(\lambda,r,\mu) \|\phi\|_r\cdot \ex^{-\lambda t},
\ea
where $|\mu|$ stands for the standard total variation measure of the finite signed measure $\mu$.

\noindent Now we apply the stability of the unperturbed system comparing $X(\cdot;0)$ 
to the following Ornstein--Uhlenbeck type process $U$ given by  
\ba
U(t) = \begin{cases} 
-\displaystyle\lambda \int_0^t U(s)\, \di s + S(t) & \mbox{ for } t\geq  0, \\ 
0 &  \mbox{ for } t\in [-r, 0).
\end{cases}
\ea 
The process $U$ has the explicit solution 
\ba
U(t) = \int_0^t \ex^{-\la (t-s)}\,\di S(s), \quad t\geq 0. 
\ea
Integration by parts then yields 
$U(t)= S(t) - \lambda \ex^{-\lambda t} \int_0^t S(s)\ex^{\lambda s}\,\di s$, such that  
\ba\label{e:U final}
|U(t)|&\leq  2\sup_{s\in[0,t]}|S(s)|.
\ea
We fix the notation $V(t) := X(t;0) - U(t)$, $t\geq -r$, and note that $X(\cdot;0) = V + U$. 
It remains to estimate $V$. Collecting the absolutely continuous parts as 
\ba
A(t) := \int_0^t \Big[\int_{[-r, 0]} U(s+u)\, \mu(\di u)\Big]\, \di s - \la \int_0^t U(s)\, \di s
\ea
we obtain the following equation for $V$ 
\ba
V(t) 
&= X(t;0) - U(t) \\
&= \int_0^t \int_{[-r, 0]} \Big(X(s+u;0) - U(s+u)\Big)\, \mu(\di u)\, \di s 
+ \int_0^t \int_{[-r, 0]} U(s+u)\, \mu(\di u)\, \di s - \la \int_0^t U(s)\, \di s\\
&= \int_0^t \int_{[-r, 0]} V(s+u)\, \mu(\di u) \, \di s + A(t), 
\ea
the solution of which has the explicit convolution representation  
\ba
\label{e:V}
V(t) = \int_0^t x^*(t-s) A'(s)\, \di s, 
\ea
where 
\ba
\label{e:A'}
A'(t) = \int_{[-r, 0]} U(t+u)\, \mu(\di u) - \la U (t). 
\ea
Equation \eqref{e:V} and inequality \eqref{e:K} then yield the estimate 
\ba
|V(t)| \leq K \int_0^t \ex^{-\la (t-s)} |A'(s)| \,\di s,\quad t\geq 0.
\ea
Furthermore, \eqref{e:A'} implies for some $C_2=C_2(\lambda,\mu,r)>0$
\ba
|A'(t)| 
&\leq |\mu|[-r, 0] \cdot  \sup_{u\in [-r, 0]} |U(t+u)|  + \la |U(t)| \\
&\leq \Big(|\mu|[-r, 0] +\la\Big) \sup_{u\in [-r, 0]} |U(t+u)|\\
&\leq C_2 \cdot \sup_{s \in [0, t]} |U(s)|.
\ea
Combining the two preceding inequalities results in a constant $C_3 = C_3(\la, K, \mu, r)>0$ such that 
\ba \label{e:V final}
|V(t)|& = K\cdot C_2  \cdot \sup_{s \in [0, t]} |U(s)| \cdot\int_0^t \ex^{-\la(t-s)}\,\di s
\leq \frac{K\cdot C_2}{\lambda}\cdot\sup_{s \in [0, t]} |U(s)| =C_3 \cdot \sup_{s \in [0, t]} |U(s)|
\ea
and \eqref{e:est1} follows as a combination of \eqref{e:C1}, \eqref{e:U final} and \eqref{e:V final}.

\smallskip
\noindent
2. We denote the perturbed fundamental solution with initial segment $\phi(0)\bI_{\{0\}}$ by $\tilde X$. It satisfies  
\ba
&\tilde X(t)=\phi(0)+\int_0^t\int\tilde X(s+u)\,\mu(\di u)\,\di s+S(t),\\
&\tilde X(t)=0,\quad t\in[-r,0).
\ea
Then
\ba
X(t;\phi)-\phi(0)x^*(t)=X(t;\phi)-\tilde X(t)+ \tilde X(t)-\phi(0)x^*(t) 
\ea
and $\tilde Y(t) :=X(t,\phi)-\tilde X(t)$ is the solution of 
\ba
&\tilde Y(t)=\int_0^t\Big[\int_{[-r,0]}\tilde Y(s+u)\,\mu(\di u)\Big]\,\di s,\\
&\tilde Y(t)=\phi(t),\quad t\in[-r,0).
\ea
Hence arguing as for the term $A$ in part 1.\ we get 
\ba
|\tilde Y(t)|\leq C_1\cdot \sup_{s\in[-r,0)}|\phi(s)|\cdot \ex^{-\lambda t}
\ea
with the same constant $C_1$ as in \eqref{e:C1}.
Analogously to the identification of $Y$ in part 1.\ we observe
\ba
\tilde X(t)-\phi(0)x^*(t)=X(t;0)=V(t)+U(t), 
\ea
where the processes $V$ and $U$ are already estimated in \eqref{e:U final} and \eqref{e:V final} of part 1. 
Inequality \eqref{e:est2} then follows by 
\ba
&|X(t,\phi)-\phi(0) x^*(t)| \leq |\tilde Y(t)|+ |V(t)|+|U(t)|
\leq C_1\sup_{[-r,0)}|\phi(s)|\ex^{-\lambda t}+2(C_3+1)\sup_{s\in[0,t]}|S(s)|.
\ea
3. Since $X(t;\phi)-x(t;\phi)=X(t;0)=V(t)+U(t)$, estimate \eqref{e:est3} follows immediately by the previous results.
\end{proof}

\bigskip 
\subsubsection{Estimates on the stochastic perturbation}

Let us rewrite the underlying L\'evy process $Z$ 
as a sum of a compound Poisson process
\ba
\eta^\rho(t) &=\sum_{s\leq t}\Delta Z(s)\bI(|\Delta Z(s)|>\rho), \qquad t\geq 0, 
\ea
whose jumps  are larger than some threshold $\rho> 1$ in absolute value, and an independent 
L\'evy process $\xi^\rho=Z-\eta^\rho$ with bounded jumps. This is always possible and can be easily seen 
by comparison of the L\'evy--Khintchine formula \eqref{e:LKZ} for $Z$ 
with those for $\eta^\rho$ and $\xi^\rho$
\ba
\label{e:LKeta}
&\ln \E \ex^{\i u \eta^\rho(t)}=t\int_{|z|>\rho} (\ex^{\i uz}-1)\,\nu(\di z),\\
&\ln \E \ex^{\i u \xi^\rho(t)}
=-\frac{\sigma^2}{2}u^2t +\i\,d_\rho u t+t\int_{|z|\leq \rho}(\ex^{\i uz}-1-\i u z)\,\nu(\di z),\\
\ea
where the new drift $d_\rho = d + \int_{1<|z|\leq \rho} z\,\nu(\di z)$. We denote the jump times and sizes of $\eta^\rho$ by $\{\tau_k\}_{k\geq 1}$,
and  $\{J_k\}_{k\geq 1}$ respectively and recall that they are independent. Moreover, the interjump times $\tau_1,\tau_2-\tau_1$,\dots, are iid
exponentially distributed random variables with the parameter
\ba
\beta_\rho=\int_{|z|>\rho}\nu(\di z),
\ea
and the jump sizes $\{J_k\}_{k\geq 1}$ are also iid with the probability law 
\ba
\P(J_k\in A)=\frac{1}{\beta_\rho}\int_{|z|>\rho}\bI_A(z)\,\nu(\di z),\quad A\in\rB(\bR).
\ea

\noindent Let $X^\e=X^\e(\cdot;\phi)$ be the solution to the delay SDDE \eqref{e:SDDE}, and
consider the stochastic integral process 
\ba\label{e:stochint}
S^{\e,\rho}(t)=\e\int_0^t F(X^\e_s)(0-)\,\di \xi^\rho(s).  
\ea

\begin{rem}
In this article we are interested in the first exit time of $X^\e(t)$ from the interval $[a, b]$, 
which coincides by definition by the first exit time of $X^\e_t$ in segment space from 
the segment interval $\ldbrack a, b\rdbrack _r$. By the Lipschitz continuity \eqref{eq:functionalLipschitz} we have 
\[
\sup_{\phi \in \ldbrack a, b\rdbrack _r} \|F(\phi)\|_r \leq  L (\|F(0)\|_r + \sup_{\phi \in \ldbrack a, b\rdbrack _r} \|\phi\|_r).
\]
That is, before the first exit $t\leq \tau$ 
the coefficient of (\ref{e:stochint}) is bounded by 
\[\|F(X^\e)(t-)\|_r \leq L(\|F(0)\|_r + \max\{|a|, b\}) < \infty.\] 
Therefore it is without a loss of generality if we assume that $F$ is uniformly bounded, that is, 
$\sup_{\phi \in D[-r, 0]} \|F(\phi)\|_r \leq C_F < \infty$ for some global constant $C_F>0$.
\end{rem}

\begin{lem}
\label{l:M}
Let $F$ be uniformly bounded by a constant $C_F>0$. 
Then for any $\rho>1$, $T>0$, $\delta>0$ and any $p>1$ there is a constant $C_S>0$ such that
for any $0\leq s\leq t\leq T$ and $\e>0$ sufficiently small we have 
\ba
\P\Big(\sup_{u\in[0,t-s]}|S^{\e,\rho}(s+u)- S^{\e,\rho}(s)  |>\delta \Big)\leq C_S\e^p. 
\ea
\end{lem}
\begin{proof}
Note that $\tilde \xi^\rho(t)=\xi^\rho(t)-d_\rho t$, $t\geq 0$, is a martingale with bounded jumps,
as well as for fixed $s\geq 0$ the stochastic integral process $\tilde S^{\e,\rho}(s+u)- \tilde S^{\e,\rho}(s)  
=\e\int_s^{s+u} F(X^\e_v)(0-)\,\di \tilde\xi^\rho(v)$, $u\geq 0$.
Then the triangle inequality, the Markov inequality, and the classical Burkholder-Davis-Gundy inequality for $p>1$ 
(see for instance \cite{Protter-04}, Theorem 4.8) 
yield a constant $C_p>0$ such that for 
$\e>0$ being sufficiently small it follows 
\ba
\P\Big(\sup_{u\in[0,t-s]}&|S^{\e,\rho}(s+u)- S^{\e,\rho}(s)  |>\delta\Big)\\
&\leq \P\Big( \e\cdot C_F\cdot |d_\rho|\cdot T>\frac{\delta}{2}\Big)
+ \P\Big(  \sup_{u\in[0,t-s]}\Big|\e\int_s^{s+u} F(X^\e_v)(0-)\,\di \tilde\xi^\rho(v) \Big|>\frac{\delta}{2}\Big)\\
&\leq 0+  \e^p\cdot \frac{2^p }{\delta^p} ~\E \sup_{u\in[s, t-s]}\Big|\int_s^{s+u} F(X^\e_v)(0-)\,\di \tilde\xi^\rho(v)\Big|^p\\
&\leq \e^p\cdot \frac{2^pC_p}{\delta^p} T^{p/2}\cdot C_F^p\cdot \Big(\sigma^2+ \int_{|z|\leq \rho} z^2\,\nu(\di z)\Big)^{p/2}
=C_S\cdot \e^p.
\ea
\end{proof}

\bigskip
\subsection{General segment Markov estimates\label{s:markov}}

Let $X$ be a c\`adl\`ag segment Markov process as given in \eqref{e:MP} 
and denote for $a<0<b$ the first exit time of $X$ from $[a, b]$ by 
\ba
\tau=\tau(\phi)=\inf\{t\geq 0\colon X(t)\notin [a,b]\}.
\ea

\begin{lem}
\label{l:b}
If for some $T\geq r>0$, $m_1$, $m_2>0$ satisfying $m_1 T\leq m_2 T<1$ and $p_1(B)$, $p_2(B)>0$ 
we have the short term estimates   
\begin{align}
&\inf_{\phi\in \ldbrack a,b\rdbrack_r}\P_\phi(\tau \leq T)\geq m_1 T,\label{e:uppershorterm1}\\
&\inf_{\phi\in \ldbrack a,b\rdbrack_r}\P_\phi(\tau \leq T, X(\tau)\in B)\geq m_1 T\cdot p_1(B),\label{e:uppershorterm2}\\
&\sup_{\phi\in \ldbrack a,b\rdbrack_r}\P_\phi(\tau \leq T, X(\tau)\in B)\leq m_2 T\cdot p_2(B),\label{e:uppershorterm3}\\
\end{align}
then for each $u\geq 0$ we have 
\begin{align}
&\sup_{\phi\in \ldbrack a,b\rdbrack_r}\P_\phi(\tau>u)
\leq \frac{\ex^{-u m_1}}{1- m_1 T},\label{e:upperlongerm1}\\
&\sup_{\phi\in \ldbrack a,b\rdbrack_r}\P_\phi(\tau>u,X(\tau)\in B)
\leq \frac{\ex^{-u m_1}}{(1- m_1 T)^2}\cdot \Big(\frac{m_2}{m_1}p_2(B)-m_1Tp_1(B)\Big),\label{e:upperlongerm2}\\
&\sup_{\phi\in \ldbrack a,b\rdbrack_r}\P_\phi(X (\tau )\in B)
\leq \frac{ m_2 }{m_1 } \cdot p_2(B) \label{e:upperlongerm3}
\end{align}
and 
\ba
\sup_{\phi\in \ldbrack a,b\rdbrack_r} \E_\phi \tau &\leq\frac{1}{m_1}.
\ea
\end{lem}
\begin{proof}
We first show \eqref{e:upperlongerm2}. For each $u>0$ denote $k=\lfloor\frac{u}{T}\rfloor$. Then $kT\leq u$ and 
\ba\label{e:dicret}\P_\phi(\tau>u, X(\tau)\in B) \leq \P_\phi(\tau> kT , X(\tau)\in B).\ea
The segment Markov property 
and \eqref{e:uppershorterm1} yield for any initial condition $\phi\in \ldbrack a,b\rdbrack_r$ 
\ba\label{e:segMark}
\P_\phi(\tau> kT , X(\tau)\in B)
&=\E_\phi \Big[ \bI(\tau> kT, X(\tau)\in B)\cdot \bI(\tau>(k-1)T)   \Big]\\
&=\E_\phi \E\Big[ \bI(\tau> kT, X(\tau)\in B)\cdot \bI(\tau>(k-1)T) \Big|\rF_{(k-1)T}  \Big]\\
&=\E_\phi \Big[ \bI(\tau>(k-1)T)\cdot  \E\Big[ \bI(\tau> kT, X(\tau)\in B)\cdot\Big|\rF_{(k-1)T}  \Big]\Big]\\
&=\E_\phi \Big[ \bI(\tau>(k-1)T)\cdot  \E_{X_{(k-1)T}}\Big[ \bI(\tau> T, X(\tau)\in B)\Big]\Big]\\
&\leq \P_\phi \Big(\tau>(k-1)T)\Big)\cdot \sup_{\psi\in \ldbrack a,b\rdbrack_r} \P_\psi(\tau> T, X(\tau)\in B)\\
&\leq \Big[1-\inf_{\phi\in \ldbrack a,b\rdbrack_r} \P_\phi(\tau\leq  T)\Big]^{k-1} 
\cdot \sup_{\psi\in \ldbrack a,b\rdbrack_r} \P_\psi(\tau> T, X(\tau)\in B)\\
&\leq (1-m_1T)^{k-1} 
\cdot \sup_{\psi\in \ldbrack a,b\rdbrack_r} \P_\psi(\tau> T, X(\tau)\in B).
\ea
While the first term satisfies 
\ba\label{e:expapprox}
(1-m_1T)^{k-1} = \frac{(1- \frac{m_1 k T}{k})^k}{1-m_1T} \leq \frac{(1- \frac{m_1 u}{k})^k}{1-m_1T} \leq \frac{\ex^{-m_1 u}}{1-m_1T}.
\ea
The last inequality is given by $\ex^{-m_1 u} - (1-m_1 u/k)^k\geq 0$ for all $m_1 u\leq k$. 
The latter condition is satisfied since $m_1 k T \leq m_1 u \leq k$, which is a consequence of $m_1 T < 1$ in the statement. 
We rewrite the last term on the right side of \eqref{e:segMark} as 
\ba\label{e:rewrite}\P_\phi(\tau> T, X(\tau)\in B) = \P_\phi(X(\tau)\in B) - \P_\phi(\tau\leq T, X(\tau)\in B).\ea 
While the second term is estimated by \eqref{e:uppershorterm2} 
we calculate the first summand, which is \eqref{e:upperlongerm3}, 
with the segment Markov property and \eqref{e:uppershorterm1} 
\ba\label{e:location}
\P_\phi(X(\tau)\in B)&=\sum_{k=1}^\infty \P_\phi(X (\tau)\in B,(k-1)T <\tau \leq kT)\\
&=\sum_{k=1}^\infty \E_\phi \E \Big[ \bI((k-1)T <\tau, X(\tau )\in B,\tau \leq kT)\Big|\rF_{(k-1)T}\Big]\Big]\\
&=\sum_{k=1}^\infty \E_\phi \E \Big[ \bI((k-1)T <\tau, X(\tau )\in B,\tau^\e\leq kT)\Big|\rF_{(k-1)T}\Big]\Big]\\
&=\sum_{k=1}^\infty \E_\phi \bI((k-1)T <\tau)\cdot \P_{X_{(k-1)T}} (X(\tau )\in B,\tau^\e\leq T)\Big]\\
&\leq \sup_{\psi\in\ldbrack a,b\rdbrack_r} \P_\psi (X(\tau )\in B,\tau\leq T)\cdot \sum_{k=1}^\infty \P_\phi((k-1)T <\tau)\\
&\leq  m_2\cdot T\cdot p_2(B)  \cdot  \sum_{k=1}^\infty \Big(1- m_1T\Big)^{k-1}\\
&=\frac{ m_2 }{m_1 } \cdot p_2(B).
\ea
The estimates \eqref{e:rewrite} and \eqref{e:location} yield  
\ba\label{e:lastterm}
 \P_\phi(\tau> T, X(\tau)\in B) \leq m_2T\cdot p_2(B)\cdot \frac{1}{m_1T} -m_1Tp_1(B) 
\ea
and \eqref{e:dicret}-\eqref{e:expapprox} with \eqref{e:lastterm} imply \eqref{e:upperlongerm2}.
Setting $B=\bR$ estimate \eqref{e:upperlongerm1} follows directly from \eqref{e:dicret}-\eqref{e:expapprox}.
Eventually 
\ba
\E_\phi \tau &\leq T \sum_{k=1}^\infty k\cdot \P_\phi((k-1)T <\tau\leq kT)=T \sum_{k=0}^\infty k\cdot \P_\phi( \tau >Tk)\leq T \sum_{k=0}^\infty (1-m_1T)^k =\frac{1}{m_1}. 
\ea
\end{proof}

\begin{lem}
\label{l:a}
Let $B$ be a Borel set. If for some $T\geq r>0$, $m_1, m_2>0$ satisfying $m_1 T \leq m_2 T <1$, 
$p_1(B), p_2(B) \leq 1$ and $\delta\in(0,|a|\wedge b)$ we have the short term estimates 
\begin{align}
&\sup_{\|\phi\|_r\leq \delta} \P_\phi(\tau \leq T,X(\tau)\in B )\leq m_1 T\cdot p_1(B),\label{e:lowershorterm1}\\
&\inf_{\|\phi\|_r\leq \delta} \P_\phi(\tau \leq  T , X(\tau)\in B)\geq m_2 T\cdot p_2(B),\label{e:lowershorterm2}\\
&\inf_{\|\phi\|_r\leq \delta} \P_\phi(\tau > T , \|X_T\|_r\leq \delta)\geq 1-m_1 T,\label{e:lowershorterm3}
\end{align}
then for each $u\geq 0$ we get 
\begin{align}
&\inf_{\|\phi\|_r\leq \delta}\P_\phi(\tau>u,X(\tau)\in B)\geq (1-m_1 T)\ex^{-\frac uT\ln(1- m_1 T)}\cdot \Big(\frac{m_2}{m_1}\cdot p_2(B) - m_1 T p_1(B)\Big),\label{e:lowerlongerm1}\\
&\inf_{\|\phi\|_r\leq \delta}\P_\phi(X(\tau)\in B)\geq \frac{m_2}{m_1} p_2(B).\label{e:lowerlongerm2}\\
&\inf_{\|\phi\|_r\leq \delta}\E_\phi \tau \geq \frac{1-m_1 T}{m_1 }.\label{e:lowerlongerm3}
\end{align}
\end{lem}
\begin{proof}
For each $u>0$ denote $k=\lceil\frac{u}{T}\rceil$. In particular $(k-1)T < u \leq kT$. Then
the segment Markov property and \eqref{e:lowershorterm3} 
yield for any initial condition $\|\phi\|_r\leq \delta$
\ba
\label{e:chain1}
\P_\phi(\tau>u) \geq \P_\phi(\tau> kT) &\geq \E_\phi \Big[ \bI(\tau> kT)\cdot \bI(\tau>(k-1)T)\cdot \bI( \|X_{(k-1)T}\|_r\leq \delta)   \Big]\\
&= \E_\phi \E \Big[ \bI(\tau> kT)\cdot \bI(\tau>(k-1)T)\cdot \bI( \|X_{(k-1)T}\|_r\leq \delta)  \Big|\rF_{(k-1)T}\Big]\\
&= \E_\phi \Big[  \bI(\tau>(k-1)T)\cdot \bI( \|X_{(k-1)T}\|_r\leq \delta) \E \Big[ \bI(\tau> kT) \Big|\rF_{(k-1)T}\Big]\Big]\\
&= \E_\phi  \bI(\tau>(k-1)T)\cdot \bI( \|X_{(k-1)T}\|_r\leq \delta) \E_{X_{(k-1)T}} \Big[ \bI(\tau> kT)\Big]\\
&\geq  \P_\phi \Big(\tau>(k-1)T,\|X_{(k-1)T}\|_r\leq \delta  \Big)
\cdot \inf_{\|\psi\|_r\leq \delta} \P_\psi\Big(\tau> T\Big)\\    
&\geq 
\Big[\inf_{\|\psi\|_r\leq \delta} \P_\psi\Big(\tau> T,\|X_T\|_r\leq \delta \Big)\Big]^{k-1}
\cdot \inf_{\|\psi\|_r\leq \delta} \P_\psi \Big(\tau>T \Big)    \\ 
&\geq (1-m_1T)^{k}
\ea
and \eqref{e:lowerlongerm3} follows by 
\ba
\E_\phi \tau &\geq T \sum_{k=1}^\infty (k-1)\cdot \P_\phi((k-1)T <\tau\leq kT)\\
&=T \sum_{k=1}^\infty \P_\phi( \tau >Tk)\geq T \sum_{k=1}^\infty (1-m_1T)^k =\frac{1-m_1T}{m_1} . 
\ea
Furthermore, inequality \eqref{e:lowerlongerm2} is a result from 
\ba
\P_\phi(X(\tau)\in B)&=\sum_{k=1}^\infty \P_\phi(X (\tau)\in B,(k-1)T <\tau \leq kT)\\
&\geq \sum_{k=1}^\infty 
\E_\phi \Big[ \bI(\tau> (k-1)T )\cdot\bI(\|X_{(k-1)T}\|_r\leq \delta) \E \Big[\bI(X(\tau )\in B,\tau \leq kT)\Big|\rF_{(k-1)T}\Big]\Big]\\
&\geq \inf_{\|\psi\|_r\leq \delta } \P_\psi \Big( X(\tau )\in B,\tau\leq T\Big)
\cdot \sum_{k=1}^\infty \P_\phi\Big( \tau>(k-1)T,\|X_{(k-1)T}\|_r\leq \delta \Big)  \\
&\geq m_2T \cdot p_2(B)\cdot \sum_{k=1}^\infty (1-m_1T)^{k-1}\\
&=\frac{ m_2 }{m_1 } \cdot p_2(B). 
\ea
Finally, repeating the chain of inequalities \eqref{e:chain1} with the additional event $\{X(\tau)\in B\}$, we get 
\ba
\P_\phi(\tau>u,~&X(\tau)\in B)\geq \P_\phi(\tau> kT ,X(\tau)\in B)\\
&\geq \E_\phi \Big[ \bI(\tau> kT,X(\tau)\in B)\cdot \bI(\tau>(k-1)T)\cdot \bI( \|X_{(k-1)T}\|_r\leq \delta)   \Big]\\
&= \E_\phi \E \Big[ \bI(\tau> kT,X(\tau)\in B)\cdot \bI(\tau>(k-1)T)\cdot \bI( \|X_{(k-1)T}\|_r\leq \delta)  \Big|\rF_{(k-1)T}\Big]\\
&\geq  \P_\phi \Big(\tau>(k-1)T,\|X_{(k-1)T}\|_r\leq \delta  \Big)
\cdot \inf_{\|\psi\|_r\leq \delta} \P_\psi\Big(\tau> T,X(\tau)\in B \Big)\\    
&\geq 
\Big[\inf_{\|\psi\|_r\leq \delta} \P_\psi\Big(\tau> T,\|X_T\|_r\leq \delta \Big)\Big]^{k-1}
\cdot \inf_{\|\psi\|_r\leq \delta} \P_\psi \Big(\tau>T,X(\tau)\in B\Big)    \\ 
&\geq (1-m_1T)^{k}\cdot \Big( \frac{m_2}{m_1}\cdot p_2(B)- m_1T\cdot p_1(B)   \Big),
\ea
and \eqref{e:lowerlongerm1} follows directly.
 \end{proof}

\subsection{Proof of the main estimates\label{s:Main estimates}}

The main result will follow directly from the following six inequalities.

\begin{lem}
\label{l:i1}
For any $\varkappa>0$  there is  $T>0$ and $\e_0>0$ such that for all $\e\in(0,\e_0]$
\begin{align}
\label{e:i1}
&\inf_{\phi\in \ldbrack a,b\rdbrack_r}\P_\phi(\tau^\e \leq T)\geq \lambda_\e \bar\nu (E) T(1-\varkappa).
\end{align}
\end{lem}

\begin{lem}
\label{l:i2}
For any $\varkappa>0$ there are $\delta>0$, $T>0$, and $\e_0>0$ such that for all $\e\in(0,\e_0]$
\begin{align}
\label{e:i2}
&\inf_{\|\phi\|_r\leq \delta} \P_\phi(\tau^\e > T , \|X_T^\e\|_r\leq \delta)\geq 1 - \lambda_\e \bar\nu (E) T(1+\varkappa).
\end{align}
\end{lem}

\begin{lem}
\label{l:i3}
For any $\varkappa>0$ and $v>b$ there is $T>0$ and $\e_0>0$ such that for all $\e\in(0,\e_0]$
\begin{align}
\label{e:i3}
&\inf_{\phi\in \ldbrack a,b\rdbrack_r}\P_\phi(\tau^\e \leq T, X^\e(\tau^\e)>v)\geq \lambda_\e T\cdot \bar \nu(E_b^j(v))(1-\varkappa).
\end{align}
\end{lem}

\begin{lem}
\label{l:i4}
For any $\varkappa>0$ and $v>b$ there is $T>0$ and $\e_0>0$ such that for all $\e\in(0,\e_0]$
\begin{align}
\label{e:i4}
&\sup_{\phi\in \ldbrack a,b\rdbrack_r}\P_\phi(\tau^\e \leq T, X^\e(\tau^\e)>v)\leq \lambda_\e T\cdot \bar \nu(E_b^j(v))(1+\varkappa).
\end{align}
\end{lem}

Analogously to Lemmas \ref{l:i3} and \ref{l:i3} one proves the estimate for the exit into the set $(-\infty,a)$.
\begin{lem}
 For any $\varkappa>0$ and $v<a$ there is $T>0$ and $\e_0>0$ such that for all $\e\in(0,\e_0]$
\begin{align}
\label{e:i5}
&\inf_{\phi\in \ldbrack a,b\rdbrack_r}\P_\phi(\tau^\e \leq T, X^\e(\tau^\e)<v)\geq \lambda_\e T\cdot \bar \nu(E_a^j(v))(1-\varkappa),\\
\label{e:i6}
&\sup_{\phi\in \ldbrack a,b\rdbrack_r}\P_\phi(\tau^\e \leq T, X^\e(\tau^\e)<v)\leq \lambda_\e T\cdot \bar \nu(E_a^j(v))(1+\varkappa).
\end{align}
\end{lem}

Before passing to the proof of the Lemmas we make several preparatory comments. Let $\varkappa>0$ be an arbitrary small number.

1. For $\gamma > -\min\{|a|,b\}$ denote 
\ba
E(\gamma)=\{z\colon \exists\, t\geq 0\text{ such that } z \cdot F_0 \cdot x^*(t)\notin [a-\gamma,b+\gamma] \}
=[e_-(\gamma),e_+(\gamma)]^c,\quad E(0)=E.
\ea
Assume that $\bar\nu(E)>0$.
Due to the continuity of the fundamental solution,
\ba
\lim_{\gamma \to 0} e_\pm(\gamma)=e_\pm
\ea
and for any $\varkappa>0$ with the help of \eqref{e:rv} we get
\ba 
\label{e:nukappa}
\frac{\nu(E(\gamma)/\e)}{\lambda_\e}= 
 \frac{\nu(E(\gamma)/\e)}{\lambda_\e \bar \nu(E(\gamma))}\cdot \frac{\bar \nu(E(\gamma))}{\bar\nu(E)}\geq 1-\frac{\varkappa}{20}
\ea
for $\gamma$ and $\e$ sufficiently small. 

\noindent
2. Let $\lambda\in(0,\Lambda)$ and $K=K(\lambda)$ according to \eqref{e:K} be fixed.
For any $\delta>0$ we can choose $R>r$ such that 
\ba
\max\{|a|,b\}\cdot K\cdot \ex^{-\lambda (R-r)}< \delta.
\ea
In particular, for any $\phi\in \ldbrack a,b\rdbrack_r$ and $t\geq R$ we have.
\ba
\label{e:x-phi-delta}
|x(t,\phi)|< \delta. 
\ea
Note that $R$ also bounds the time horizon of a non-normal growth exit so that a deterministic
exit can occur only before the time instant $R$.

\noindent
3. For $\delta$ and $R$ chosen, we can fix $T>0$ such that
\ba
\frac{2R}{T}\leq \frac{\varkappa}{20}.
\ea
4. Finally, for $\delta>0$ and $0\leq s\leq t\leq T$ denote 
\ba
\cE_{s,t}=\cE_{s,t}(\delta)=\Big\{\sup_{ u\in[0,t-s]}   |S^{\e,\rho}(s+u)-S^{\e,\rho}(s)  |\leq \delta\Big\}.
\ea
With the help of Lemma \ref{l:M} with $p>\alpha$ we get $\P(\cE_{s,t}(\delta))=o(\lambda_\e)$, and in particular
\ba
\label{e:calE}
\P(\cE_{s,t}(\delta))> 1-\frac{\varkappa}{20}T \lambda_\e\bar\nu(E)
\ea
for $\e$ small enough.

\subsubsection{Proof of Lemma \ref{l:i1}}

We show that with high probability the exit from $[a,b]$ is to occur imminently after the first large jump. 


For $\delta=\delta(\varkappa,\gamma)>0$, $T=T(\delta)$ and $R=R(\delta)$ to be chosen later and for any $\phi\in \ldbrack a,b\rdbrack_r$ 
we exclude the following error events which eventually turn out to have 
small probability and estimate
\ba
\label{e:p1}
\P_\phi(\tau^\e \leq T)&\geq \P_\phi(\tau^\e \leq T,\e J_1 \in E(\gamma), 
R\leq \tau_1\leq T-R,\tau_2>T,\cE_{0,\tau_1}(\delta),\cE_{\tau_1,T}(\delta)).
\ea
Now we show that with a proper choice of the parameters, the set of conditions 
$\{\e J_1 \in E(\gamma)\}\cap \{R\leq \tau_1\leq T-R\}\cap\{\tau_2>T\}\cap\cE_{0,\tau_1}\cap\cE_{\tau_1,T}$
will imply that the exit occurs before the time $T$, $\tau^\e\leq T$. 

Indeed, at the time instant $\tau_1$ we have
\ba
X^\e(\tau_1)=X^\e(\tau_1-)+ \e F(X^\e_{\tau_1})(0-) J_1 .
\ea
To estimate $X^\e(\tau_1-)$ we note that \eqref{e:est1} together with \eqref{e:x-phi-delta} 
guarantee that for all $\delta>0$, on the event $\cE_{0,\tau_1}(\delta)$ we have
\ba
\sup_{t\in[\tau_1-r,\tau_1)}|X^\e(t,\phi)|< 2 C \delta.
\ea
Then obviously if $2LC\delta<F_0/2$ then the event
\ba
\e |J_1| (F_0-2 L C \delta)\geq 2\max\{|a|,b\}
\ea
implies that $\tau^\e=\tau_1\leq T$.
Hence from now on we assume without loss of generality that $\e J_1\in E(\gamma)$ and $\e|J_1|\leq C_J$ for 
$C_J=4 \max\{|a|,b\}/F_0$.
Then according to \eqref{e:est2} on the event $\cE_{\tau_1,T}(\delta)$
\ba
\sup_{t\in[\tau_1,\tau_1+R]}|X^\e(t,\phi)-  X^\e(\tau_1)\cdot x^*(t-\tau_1)   |
\leq C \cdot( 2C\delta + \delta )=(2C+1)C\delta.
\ea
Finally, comparing
\ba
|X^\e(\tau_1) - \e F_0 J_1|\leq |X^\e(\tau_1-) |+ \e |J_1| \cdot |F(X^\e_{\tau_1})(0-)-F_0 | 
\leq  2 C \delta+ \e |J_1| \cdot  L \cdot 2C\delta
\ea
we obtain that 
\ba
|X^\e(\tau_1)\cdot x^*(t )- \e J_1 F_0\cdot x^*(t)|\leq K|X^\e(\tau_1)- \e J_1F_0| \leq  2 CK \delta+ \e |J_1| \cdot 2KLC\delta.
\ea
This means that  if $\e J_1\in E(\gamma)$ and $\e|J_1|\leq C_J$ then either 
\ba
\sup_{t\in[\tau_1,\tau_1+R]}X^\e(t,\phi) &\geq \sup_{t\in[0,R]}X^\e(\tau_1)\cdot x^*(t) -(2C+1)C\delta\\
&\geq  \sup_{t\in[0,R]} \e J_1 F_0\cdot x^*(t) - 2 CK \delta- \e |J_1| \cdot 2KLC\delta   -(2C+1)C\delta\\
&\geq b+\gamma -C\Big(2 K - 2C_J 2KL  +2C+1\Big)\delta\geq b+\frac{\gamma}{2}
\ea
for $\delta>0$ small enough
or analogously
\ba
\inf_{t\in[\tau_1,\tau_1+R]}X^\e(t,\phi) \leq a-\frac{\gamma}{2}
\ea
for the same $\delta$. Hence, from now on $\delta>0$ as well as $R$ and $T$ are fixed. 

It is left to show that estimate \eqref{e:p1} yields the required accuracy in the limit $\e\to 0$.
First we choose the large jump threshold $\rho>0$ such that $\ex^{-\beta_\rho T(\delta)}\geq 1-\varkappa/20$.

Hence for $\e$ small
\ba
\label{e:abo}
\P_\phi(\tau^\e\leq T)&\geq \P_\phi( \e J_1 \in E(\gamma), R(\delta)\leq \tau_1\leq T(\delta)-R(\delta),\tau_2>T(\delta)
,\cE_{0,\tau_1}(\delta) ,\cE_{\tau_1,T}(\delta) )\\
&\geq \P( \e J_1 \in E(\gamma))\cdot \P(R(\delta)\leq \tau_1\leq T(\delta)-R(\delta),\tau_2>T(\delta))
- 2\frac{\varkappa}{20}T(\delta)\lambda_\e \bar \nu(E)   \\
&\geq \frac{\nu(E(\gamma)/\e)}{\beta_\rho}\cdot
\int_{R(\delta)}^{T(\delta)-R(\delta)} 
\int_{T(\delta)-t_1}^\infty \beta_\rho^2\cdot \ex^{-t_1 \beta_\rho}\cdot\ex^{-\beta t_2}\,\di t_2\,\di t_1
-\frac{\varkappa}{10}T(\delta)\lambda_\e \bar \nu(E)  \\
&=\nu(E(\gamma)/\e) \cdot (T(\delta)-2R(\delta))\ex^{-\beta_\rho T(\delta)}   -\frac{\varkappa}{10}T(\delta)\lambda_\e\bar \nu(E)\\
&=T(\delta)\lambda_\e  \bar \nu(E) \frac{\nu(E(\gamma)/\e)}{\lambda_\e\bar \nu(E)} \cdot \frac{T(\delta)-2R(\delta)}{T(\delta)}\ex^{-\beta_\rho T(\delta)} 
-\frac{\varkappa}{10}T(\delta)\lambda_\e\bar \nu(E)\\
&=T(\delta)\lambda_\e \bar\nu(E) \cdot\Big(1-\frac{\varkappa}{20}\Big)^3 
-\frac{\varkappa}{10}T(\delta)\lambda_\e\bar \nu(E)\\
&\geq  T(\delta)\lambda_\e\bar \nu(E)\cdot  (1-\varkappa).
\ea

\bigskip 

\subsubsection{Proof of Lemma \ref{l:i2}}
 
Passing to the complements we get
\ba
\P_\phi \Big(\tau^\e\leq T\text{ or }\|X_{T}^\e\|_r> \delta \Big)
\leq \P_\phi \Big(\tau^\e\leq T\Big)+\P_\phi \Big(\tau^\e> T\text{ and }\|X^\e_{T}\|_r> \delta \Big).
\ea
\textbf{Step 1.} We consider the following decomposition
Then
\ba
\P_\phi  (\tau^\e\leq T )
&= \P_\phi(\tau^\e\leq T,\tau_1 > T)
+\P_\phi(\tau^\e\leq T,\tau_1\leq T<\tau_2)
+\P_\phi(\tau^\e\leq T,\tau_2\leq T)\\
&=p_1+p_2+p_3
\ea
and show that the main contribution to the exit probability is made by the first (and virtually only the first) jump $J_1$, i.e.\ by
the term $p_2$.

\noindent
1. To estimate $p_1$ we write
\ba
p_1&\leq \P_\phi(\tau^\e\leq T,\tau_1 > T,\cE_{0,T})+\P(\cE_{0,T}^c)=p_{11}+p_{12}\\
\ea
On the event $\{\tau_1>T\}\cap \cE_{0,T}(\delta)$, due to \eqref{e:est1} in Lemma \ref{l:est}
\ba
\sup_{t\in [-r,T]}|X^\e(t,\phi)|\leq 2C\delta,
\ea
which is incompatible with $\{\tau^\e\leq T\}$ for $\delta$ sufficiently small, and hence
$p_{11}=0$. By Lemma \ref{l:M}, $p_{12}\leq T\lambda_\e\bar\nu(E)\varkappa/20$ for $\e$ small enough.

\noindent 
2. Show that the probability $p_2$ is the essential one. Take into account that $\tau_1$ and 
$\tau_2-\tau_1$ are i.i.d. exponentially distributed r.vs.\ with
the parameter $\beta_\rho:=\int_{|z|>\rho} \nu(\di y)$ ($\rho$ will be chosen large to guarantee that $\beta_\rho$ is small).
Note that $\Law(\tau_1|\tau_1\leq T<\tau_2)$ is uniform on $[0,T]$, see e.g.\ \cite[Proposition 3.4]{Sato-99}.
Then we decompose and disintegrate
\ba
\P_\phi(\tau^\e\leq T,\tau_1\leq T<\tau_2)
&\leq\P_\phi\Big(\tau^\e\leq T,\tau_1\leq T<\tau_2,\cE_{0,\tau_1} ,\cE_{\tau_1,T} \Big)
+\frac{1}{T}\int_0^T\Big(\P_\phi(\cE_{0,t}^c)+\P_\phi(\cE_{t,T}^c)\Big)\,\di t\\
&=
\P_\phi\Big(\tau^\e\leq T,R\leq \tau_1\leq T-R, \tau_2>T,\cE_{0,\tau_1} ,\cE_{\tau_1,T} \Big)\\
&+\P_\phi\Big(\tau^\e\leq T,\tau_1< R< T<\tau_2,\cE_{0,\tau_1} ,\cE_{\tau_1,T} \Big)\\
&+\P_\phi\Big(\tau^\e\leq T,T-R< \tau_1\leq  T<\tau_2,\cE_{0,\tau_1} ,\cE_{\tau_1,T} \Big)
+ p_{24} \\
&=p_{21}+p_{22}+p_{23} +p_{24}.
\ea
a) We show that for any $\gamma>0$ small enough we can choose $\delta>0$ such that
on the event $\{R\leq \tau_1\leq T-R,\tau_2>T\}\cap \cE_{0,\tau_1}(\delta)\cap \cE_{\tau_1,T}(\delta)$ we have
\ba
\{\tau^\e\leq T\}\subseteq \{\e J_1\in E(-\gamma)\}
\ea
or equivalently we show that if the jump size $\e J_1$ is not large enough, no exit occurs.

The set of conditions in the probability $p_{11}$ guarantees with the help of \eqref{e:est1}
that 
\ba
\sup_{[-r,\tau_1)}|X^\e(t;\phi)|\leq 2C\delta.
\ea
At the time instant $\tau_1$ we have
\ba
X^\e(\tau_1)=X^\e(\tau_1-)+ \e F(X^\e_{\tau_1})(0-) J_1 
\ea
and 
\ba
\label{e:F0}
|F(X^\e_{\tau_1})(0-)-F_0|\leq L \cdot 2C\delta
\ea
For $t\in[\tau_1,T]$
\ba
|X^\e(t;\phi) - \e J_1 F_0\cdot x^*(t-\tau_1)| &\leq |X^\e(t;\phi)-X^\e(\tau_1)\cdot x^*(t-\tau_1)|\\
&+ 
|X^\e(\tau_1) - \e J_1 F_0  | \cdot |x^*(t-\tau_1)|
\ea
Then according to \eqref{e:est2} on the event $\cE_{\tau_1,T}(\delta)$
\ba
\label{e:c1}
\sup_{t\in[\tau_1,T]}|X^\e(t,\phi)-  X^\e(\tau_1)\cdot x^*(t-\tau_1)   |
\leq C \Big( 2C\delta + \delta   \Big)=(2C+1)C\delta.
\ea
Finally, noting that if $\e J_1\notin E(-\gamma)$, then $|\e J_1|\leq \tilde C_J=2\max\{|a|,b\}$ we compare
\ba
\label{e:c2}
|X^\e(\tau_1) - \e F_0 J_1|\leq |X^\e(\tau_1-) |+ \e |J_1| \cdot |F(X^\e_{\tau_1})(0-)-F_0 | 
\leq  2 C( 1 + \tilde C_J   L )\delta.
\ea
Hence combining \eqref{e:c1} and \eqref{e:c2} we obtain 
\ba
\sup_{[\tau_1,T]}|X^\e(t;\phi) - \e J_1 F_0\cdot x^*(t-\tau_1)| \leq \tilde C\delta
\ea
for some $\tilde C>0$. Choosing $\delta$ small such that $\tilde C\delta<\gamma/2$ we get that if $\e J_1\notin E(-\gamma) $ 
then either 
\ba
\sup_{t\in[\tau_1,T]}X^\e(t,\phi) &\leq \sup_{t\in[0,T]} \e J_1 F_0\cdot x^*(t) +\tilde C\delta \leq b-\frac{\gamma}{2}.
\ea
or analogously
\ba
\inf_{t\in[\tau_1,T]}X^\e(t,\phi) \geq a+\frac{\gamma}{2}.
\ea
Now choose $\rho$ such that $\ex^{-\beta_\rho T(\delta)}\geq 1-\varkappa/20$. 
Then analogously to \eqref{e:abo} for small $\e$ we get
\ba
p_{21}&\leq \P\Big(
\e J_1\in E(-\gamma),R\leq \tau_1\leq T-R \Big)
=\P(\e J_1\in E(-\gamma))\cdot \P(R\leq \tau_1\leq T-R )\\
&\leq 
\frac{\nu(E(-\gamma)/\e)}{\beta_\rho}\cdot
\beta_\rho (T-2R)
\leq \lambda_\e T(\delta)\bar \nu(E)\Big(1+\frac{\varkappa}{20}\Big)
\ea

\noindent 
b)
To estimate the probability $p_{22}$ we note that on $[0,\tau_1)$
the process $X^\e(\cdot,\phi)$ belongs to the $2C\delta$-neighborhood of zero.
Recalling \eqref{e:F0} we choose $\delta>0$ small enough such that $|F(X_{\tau_1}^\e)(0-)|\leq 2|F_0|$
Hence, to guarantee the exit, the jump size $J_1$ must obviously satisfy 
$2\e\cdot |F_0|\cdot |J_1|\geq \delta'$ for some $\delta'>0$, since otherwise applying \eqref{e:est1} 
to the perturbation process
\ba
S(t)=\e \int_0^t F(X_{s-}^\e)(0-)\,\di\xi^\rho(s) + \e F(X_{\tau_1}^\e)(0-) J_1\cdot \bI_{[\tau_1,\infty)}(t)
\ea
which satisfies on $\cE_{0,\tau_1}\cap \cE_{\tau_1,T}$ 
\ba
\sup_{t\in[0,T]}|S(t)|< 2\delta+ \delta'
\ea
we obtain that the $X^\e(\cdot)$ does not leave the $C(3\delta+\delta')$-neighborhood of zero on $[0,T]$
so the exit would be impossible.
Hence choosing $T$ large we obtain for $\e$ small enough that
\ba
\label{e:C'}
p_{22}\leq \P\Big(\e |J_1|\geq \frac{\delta'}{2|F_0|},\tau_1\leq R\Big)
&= \frac{1}{\beta_\rho}\int_{|z|\geq \frac{\delta'}{2\e|F_0| }}\nu(\di z)\cdot \beta_\rho 
\cdot \int_0^R \ex^{-\beta_\rho t}\,\di t\\
&\leq 2C' \cdot \frac{R}{T} \cdot \lambda_\e T\leq \frac{\varkappa}{20} \cdot \bar\nu(E)\cdot \lambda_\e T,
\ea
where we used that due to \eqref{e:rv}
\ba
\lim_{\e\to 0}\frac{1}{\lambda_\e}\int_{|z|\geq \frac{\delta'}{2\e|F_0| }}\nu(\di z)= 
C'=\bar\nu\Big(|z|> \frac{\delta'}{2|F_0|} \Big)\in (0,\infty).
\ea

\noindent
c) Analogously to b) we estimate the probability $p_{23}$: 
\ba
p_{23}\leq  \P\Big(\e |J_1|\geq \frac{\delta'}{2|F_0|},T-R\leq \tau_1\leq T\Big)
&\leq 2C' \cdot \frac{R}{T} \cdot \lambda_\e T\leq \frac{\varkappa}{20}\cdot \bar\nu(E) \cdot \lambda_\e T.
\ea
d) Obviously due to Lemma \ref{l:M}
\ba
p_{24}\leq 2\frac{\varkappa}{20}\cdot \bar\nu(E) \cdot  T \cdot \lambda_\e.
\ea
 
\noindent 
3. To estimate $p_3$ we argue analogously that either the first large jump $J_1$ has to 
satisfy $2\e\cdot |F_0|\cdot |J_1|>\delta'$ or, if this is not the case, the 
second jump $J_2$ has to satisfy the same condition, since otherwise the solution $X^\e(\cdot;\phi)$ will stay in 
some small neighborhood of zero which size is depends on $\delta$, $\delta'$.
Together with the condition that at least two jumps occur on $[0,T]$ this yields
\ba
p_3&\leq\P_\phi\Big( \tau^\e\leq T,\tau_2\leq T ,\cE_{0,\tau_1}(\delta),\cE_{\tau_1,\tau_2}(\delta), \cE_{\tau_2,T }(\delta)  \Big)
+\P_\phi(\cE_{0,T}^c(3\delta))\\
&\leq 2 \P_\phi(2\e\cdot |F_0|\cdot | J_1|>\delta',\tau_2\leq T\Big)+\P(\cE_{0,T}^c(3\delta_1))\\
&\leq \frac{2}{\beta_\rho}\int_{|z|\geq \frac{\delta'}{2 |F_0|\e}}\nu(\di z)\cdot    \beta_\rho^2  
\cdot T
\leq 4C'\cdot \beta_\rho\cdot  T \cdot \lambda_\e.
\ea
We choose $\rho>1$ such that 
$4C'\beta_\rho<\bar\nu(E)\cdot \varkappa/20$ to obtain $p_2\leq \varkappa\cdot \bar\nu(E)\cdot T \cdot \lambda_\e/20$.

\noindent
\textbf{Step 2}.
Now we estimate
\ba
\P_\phi \Big(\tau^\e> T\text{ and }\|X^\e_{T}\|_r> \delta \Big)
&=\P_\phi \Big(\tau^\e> T\text{ and }\|X^\e_{T}\|_r> \delta ,\tau_1>T\Big)\\
&+\P_\phi \Big(\tau^\e> T\text{ and }\|X^\e_{T}\|_r> \delta,\tau_1\leq T< \tau_2 \Big)\\
&+\P_\phi \Big(\tau^\e> T\text{ and }\|X^\e_{T}\|_r> \delta,\tau_2\leq T \Big)=q_1+q_2+q_3.
\ea

Since we have to control the behavior of $X^\e$ at the end segment of the interval $[0,T]$, we have to exploit 
the exponential stability of the deterministic delay system in order to inhibit the uncontrolled accumulation of small errors 
over time.

\noindent 
1. Let $\delta>0$ be chosen in Step 1 and fixed. We choose $\delta_1<\frac{\delta}{2C}$
and $T$ large such that $C\cdot \ex^{-\lambda(T-r)}<\frac12$ and then \eqref{e:est1} 
guarantees that on the event $\cE_{0,T}(\delta_1)$, the trajectory $X^\e$ belongs to the 
$\delta$-neighborhood of zero, hence for small $\e$
\ba
q_1\leq  \P_\phi \Big(\tau^\e> T\text{ and }\|X^\e_{T}\|_r> \delta,\tau_1>T, \cE_{0,T}(\delta_1) \Big)+\P(\cE_{0,T}^c(\delta_1) )
\leq 0+ \frac{\varkappa}{20}T\lambda_\e \bar\nu(E).
\ea

\noindent
2. Estimate the probability $q_2$. Take into account that $\tau_1$ and $\tau_2-\tau_1$ are iid exponentially distributed r.vs.\ with
the parameter $\beta_\rho:=\int_{|z|>\rho} \nu(\di y)$ ($\rho$ will be chosen large to guarantee that $\beta_\rho$ is small).
Then we desintegrate
\ba
\P_\phi&( \tau^\e> T\text{ and }\|X^\e_{T}\|_r> \delta,\tau_1\leq T< \tau_2 )\\
&\leq\P_\phi\Big(\tau^\e> T\text{ and }\|X^\e_{T}\|_r> \delta,\cE_{0,\tau_1}(\delta_1),\cE_{\tau_1,T}(\delta_1)\Big)
+\P_\phi(\cE_{0,T}^c(2\delta_1))\\
&=
\P_\phi\Big(\tau^\e> T,\tau_1< R,\tau_2>T,\|X^\e_{T}\|_r> \delta,\cE_{0,\tau_1}(\delta_1),\cE_{\tau_1,T}(\delta_1)\Big)\\
&+\P_\phi\Big(\tau^\e> T,T-R< \tau_1\leq  T, \tau_2>T,\|X^\e_{T}\|_r> \delta,\cE_{0,\tau_1}(\delta_1),\cE_{\tau_1,T}(\delta_1)\Big)\\
&+\P_\phi\Big(\tau^\e> T,R\leq \tau_1\leq T-R,\tau_2>T,\|X^\e_{T}\|_r> \delta,\cE_{0,\tau_1}(\delta_1),\cE_{\tau_1,T}(\delta_1)\Big)\\
&+\P_\phi(\cE_{0,T}^c(2\delta_1))\\
&=q_{21}+q_{22}+q_{23}+q_{24}.
\ea
a) As in Step 1 b) if $2\e \cdot |F_0|\cdot |J_1|\leq \delta'$ for $\delta'>0$ and $\delta_1$ sufficiently small then
$\|X^\e_{T}\|_r\leq  \delta$. Hence,
\ba
q_{21}\leq \P\Big( 2\e \cdot |F_0|\cdot |J_1|>\delta',\tau_1< R\Big)\leq \frac{\varkappa}{20}\cdot \bar\nu(E)\cdot \lambda_\e T
\ea
as in \eqref{e:C'}.

\noindent
b) Analogously, to estimate $q_{22}$ note that right before the jump $\tau_1$, $\|X_{\tau_1-}\|_r\leq \delta/2$, so that
if $2\e \cdot |F_0|\cdot |J_1|\leq \delta'$ then $\|X^\e_{T}\|_r\leq  \delta$.
Hence
\ba
q_{22}\leq \P(2\e |F_0|\cdot|J_1|> \delta',T-R<\tau_1\leq T)\leq \frac{\varkappa}{20} \cdot \bar\nu(E) \cdot \lambda_\e T.
\ea
c) Finally, again, if $2\e \cdot |F_0|\cdot |J_1|\leq \delta'$ then $\|X^\e_{T}\|_r\leq  \delta$, so that 
we can assume that $2\e \cdot |F_0|\cdot |J_1|> \delta'$.
On the other hand, right before the jump $\tau_1$, $\|X_{\tau_1-}\|_r\leq \delta/2$. Since on the set of events of $q_{23}$
$T-r-\tau_1>T-R>T-r$, we choose the difference $R-r$ large enough such that the solution $X^\e$ satisfies 
$\|X^\e_{T}\|_r\leq  \delta$. Thus
$q_{23}=0$.

\noindent 
d) As usual, for $\e$ small enough $q_{14}\leq \varkappa\cdot \bar\nu(E) \cdot \lambda_\e T/ 20$.
\noindent 
3. To estimate $q_3$ we argue analogously that at least one of the jumps $J_1$ or $J_2$ should satisfy 
$2\e\cdot |F_0|\cdot |J_i|>\delta'$. 
Together with the condition that at least two jumps occur on $[0,T]$ this yields as in Step 1.3
\ba
q_3&\leq 2 \P_\phi(\e\|F\|\cdot | J_1|>\delta/3,\tau_2\leq T\Big)+\P_\phi(\cE_{0,T}^c(3\delta_1))
\leq \frac{\varkappa}{20} \cdot \bar\nu(E)\cdot T \cdot \lambda_\e
\ea
with the same choice of $\rho>1$.

Eventually, collecting the estimates from the Steps 1 and 2, we find
\ba
\P_\phi(\tau^\e\leq T\text{ or }\|X^\e_T\|_r>\delta)\leq T\lambda_\e\bar\nu(E)(1+\varkappa).
\ea

\subsubsection{Proof of Lemmas \ref{l:i3} and \ref{l:i4}}

The proof of Lemma \ref{l:i3} goes along the lines with the proof upper estimate for 
the probability $\P_\psi(\tau^\e\leq T)$ in Step 1 in Lemma \ref{l:i2}.
The only difference is the weaker condition on the initial segment $\psi\in \ldbrack a,b\rdbrack_r $ 
and an additional condition on the exit location $X^\e(\tau^\e)>v$ for $v>b$.
We estimate
\ba
\P_\psi (\tau^\e\leq T,X^\e(\tau^\e)>v)
&=\P_\psi (\tau^\e\leq T,\tau_1>T, X^\e(\tau^\e)>v)\\
&+\P_\psi (\tau^\e\leq T,\tau_1\leq T<\tau_2,X^\e(\tau^\e)>v)\\
&+\P_\psi (\tau^\e\leq T,\tau_2\leq T,X^\e(\tau^\e)>v)\\
&=p_1+p_2+p_3.
\ea
We consider the term $p_1$ in detail and show how to adapt the argument. Indeed, 
\ba
p_1\leq \P_\psi (X^\e(\tau^\e)>v,\tau^\e\leq T,\tau_1>T,\cE_{0,T}(\delta))+\P_\psi (\cE_{0,T}^c(\delta))=p_{11}+p_{12}.
\ea
We show that $p_{11}=0$ for $\delta$ small. Indeed, due to Lemma \ref{l:est}
\ba
\label{e:h1}
|X^\e(t,\psi)-x(t;\psi)|\leq C\delta.
\ea
For $\psi\in\ldbrack a,b\rdbrack_r$, let us consider two cases.

\noindent 
a) Let $\psi$ be such that
$\sup_{t\in[0,T]}x(t;\psi)\leq v-2C\delta$.
Hence \eqref{e:h1} implies that $\sup_{t\in[0,T]}X^\e(t;\psi)\leq v-C\delta <v$ and hence $X^\e(\tau^\e)\leq v-C\delta<v$.

\noindent
b) On the other hand, assume that there is $t^*$ such that
\ba
t^*=t^*(\psi)=\inf\{t\geq 0\colon x(t;\psi)> v-2C\delta\}\leq T.
\ea
Hence, applying \eqref{e:h1} again we get 
$X^\e(t^*;\psi)\geq v-3C\delta >b$ and thus $\tau^\e< t^*$ and $X^\e(\tau^\e)\leq v-C\delta<v$.

Therefore, $p_{11}=0$, and $p_1\leq p_{12}$ which is known to be of the order $o(\lambda_\e)$, see \eqref{e:calE}.

The probabilities $p_2$ and $p_3$ are treated analogously by taking into account that no 
exit with $X^\e(\tau^\e)>v$ can occur before or after the first jump $\tau_1$.

The proof of Lemma \ref{l:i4} goes along the lines with the proof of the lower estimate for 
the probability $\P_\psi(\tau^\e\leq T)$ in Lemma \ref{l:i1} with the same obvious modifications: the 
condition $X^\e(\tau^\e)>v$ on the exit location can be satisfied (with high probability) only if $\tau^\e=\tau_1$ and   
the jump size $\e J_1$ is large enough so that it belongs to a set $E_b(v+\gamma)$.

\subsubsection{Proof of the main Theorem \ref{t:main}}
 
Let the initial segment $\phi$ satisfy
\eqref{e:phinoexit},
in particular $\phi\in \ldbrack a,b\rdbrack_r$.

1. Combining Lemma \ref{l:b} with Lemmas \ref{l:i1}, \ref{l:i3} and \ref{l:i3}, we immediately obtain 
estimates from above for the probabilities $\P_\phi(\lambda_\e\tau^\e>u)$, $u>0$, $\P_\phi(X^\e(\tau^\e)>v)$, $v>b$, and the mean value 
$\lambda_\e\E_\phi \tau^\e$.

2. On the other hand, or each initial segment $\|\phi\|_r\leq \delta$, for $\delta>0$ small enough,
Lemmas \ref{l:a}, \ref{l:i1}, \ref{l:i3}, and \ref{l:i3}
yield the  estimates from below. 
Hence, it is left to relax the condition on the initial segment $\phi$. This can be done easily.

Let $\phi$ be an initial segment with no deterministic exit which satisfies \eqref{e:phinoexit} and let $u>0$.
Denote $\delta_\phi:=\operatorname{dist}(x(\cdot;\phi),[a,b])>0$.
Choose $R>r$ large enough and $\delta_0>0$ so that on $\cE_{0,R}(\delta_0)$ we have $\|X^\e(R;\phi)\|_r\leq \delta$. Let 
$\varkappa>0$.
Then
the segment Markov property yields
\ba
\P_\phi(\lambda_\e\tau^\e>u)&\geq \P_\phi(\lambda_\e\tau^\e>u,\cE_{0,R}(\delta_0)) 
=\E_\phi \E[\bI( \lambda_\e\tau^\e>u,\cE_{0,R}(\delta_0)  )|\rF_{R}\rdbrack \\
&=\E_\phi\bI(\cE_{0,R}(\delta_0) )\cdot \E[\bI( \lambda_\e\tau^\e>u )|\rF_{R}\rdbrack \\
&=\E_\phi[\bI(\cE_{0,R}(\delta_0) )\cdot  \P_{X_R^\e}(\lambda_\e(\tau^\e-R)>u)]\\
&\geq\inf_{\|\psi\|_r\leq \delta}\P_\psi  (\lambda_\e\tau^\e>u+\lambda_\e R) -  \P_\phi(\cE^c_{0,R}(\delta_0) )\\
&\geq \inf_{\|\psi\|_r\leq \delta}\P_\psi  (\lambda_\e\tau^\e>u)\cdot (1-\varkappa)
\ea
in the limit as $\e\to 0$.

Analogously, for any $\varkappa>0$ and $\e$ small we estimate the mean value:
\ba
\lambda_\e \E_\phi \tau^\e &\geq 
\lambda_\e \E_\phi[\tau^\e\cdot \bI(\tau_\e>R)\cdot \bI(\cE_{0,R}(\delta_0)] 
=\lambda_\e\E_\phi [\E[  \tau^\e\cdot \bI(\tau_\e>R)\cdot \bI(\cE_{0,R}(\delta_0)  )|\rF_{R}\rdbrack \\
&= \lambda_\e\E_\phi[\bI(\tau_\e>R)\cdot \bI(\cE_{0,R}(\delta_0) )\cdot \E[\tau^\e|\rF_{R}\rdbrack \\
&= \lambda_\e\E_\phi[\bI(\tau_\e>R)\cdot\bI(\cE_{0,R}(\delta_0) )\cdot  \E_{X_R^\e}(\tau^\e-R)]\\
&\geq  \lambda_\e\E_\phi[\bI(\tau_\e>R)\cdot\bI(\cE_{0,R}(\delta_0) )\cdot \inf_{\|\psi\|_r\leq \delta} \E_{\psi} (\tau^\e-R)]\\
&= \lambda_\e  \cdot \inf_{\|\psi\|_r\leq \delta} \E_{\psi} (\tau^\e-R) \cdot   \E_\phi[\bI(\tau_\e>R)\cdot\bI(\cE_{0,R}(\delta_0) )]\\
&\geq \lambda_\e  \cdot \inf_{\|\psi\|_r\leq \delta} \E_{\psi} (\tau^\e-R) 
\cdot \Big(1-\P_\phi(\tau_\e\leq R)-\P_\phi(\cE_{0,R}^c(\delta_0) )\Big)\\
& \geq  \lambda_\e  \cdot \inf_{\|\psi\|_r\leq \delta} \E_{\psi} \tau^\e-\lambda_\e R-
\lambda_\e \cdot \inf_{\|\psi\|_r\leq \delta} \E_{\psi} \tau^\e\cdot \Big(\P_\phi(\tau_\e\leq R)+\P_\phi(\cE_{0,R}^c(\delta_0) )\Big)\\
&\geq \lambda_\e (1-\varkappa) \cdot \inf_{\|\psi\|_r\leq \delta} \E_{\psi} \tau^\e.
\ea
The probability $\P_\phi(\lambda_\e\tau_\e>u, X^\e(\tau_\e)>v)$, $v>b$, is treated analogously.

%

\begin{small}

\end{small}

\end{document}